\documentclass[10pt]{amsart}

\usepackage[margin=1in]{geometry}
\usepackage{amsmath,amssymb,amsthm,mathtools}
\usepackage{enumitem}
\usepackage{xcolor}
\usepackage{tikz}
\usepackage{mathtools}
\mathtoolsset{showonlyrefs}
\definecolor{nicegreen}{RGB}{34,120,85}
\usepackage[
    colorlinks=true,
    linkcolor=nicegreen,
    citecolor=nicegreen,
    urlcolor=nicegreen
]{hyperref}
\numberwithin{equation}{section}

\newtheorem{theorem}{Theorem}[section]
\newtheorem{proposition}[theorem]{Proposition}
\newtheorem{lemma}[theorem]{Lemma}

\newcommand{\PP}{\mathbb P}

\newcommand{\calA}{\mathcal A}
\newcommand{\calD}{\mathcal D}
\newcommand{\calG}{\mathcal G}
\newcommand{\calM}{\mathcal M}
\newcommand{\calN}{\mathcal N}
\newcommand{\eps}{\varepsilon}
\newcommand{\dd}{\,\mathrm d}
\DeclareMathOperator{\tr}{tr}

\begin{document}

\title[Phase transition in two interacting urns]{The phase transition for domination in two interacting urns with power reinforcement}

\author[J. Xia]{Jiaming Xia}

\address{Center for Mathematics and Interdisciplinary Sciences, Fudan University, Shanghai 200433, China}
\address{Shanghai Institute for Mathematics and Interdisciplinary Sciences (SIMIS), Shanghai 200433, China}

\subjclass[2010]{}
\date{\today}

\begin{abstract}
We study two interacting urns with power reinforcement \(W(n)=n^\alpha\), for \(\alpha>1\). Using Qin's stochastic-approximation theorem as input, we reduce the phase transition problem for domination to the sign analysis of a single analytic function \(\Phi_p\). This deterministic reduction identifies the sharp threshold and proves that the two natural critical parameters for domination coincide. We determine the first-order behavior of the critical interaction parameter as $\alpha\downarrow1$ and show that it is not nondecreasing in the reinforcement exponent, answering two questions posed by Qin.
\end{abstract}
\maketitle

\section{Background and motivation}

Reinforced urn processes are among the simplest probabilistic models where history considerably influences future dynamics. In the classical Pólya urn, drawing a color makes that color more likely to be drawn later. This feedback mechanism has been used as a tractable model for learning, competition, preferential attachment, market share, and collective decision making (for example, see \cite{Pemantle2007}). In a single strongly reinforced urn, the feedback is so strong that eventually only one color is drawn.

Interacting urns introduce a second effect. Each urn is reinforced, sometimes from a source of its own local composition and sometimes from the composition of the whole system. When the parameter \(p\) is small, different urns may preserve different local preferences. When \(p\) is large, the system tends to synchronize and one color may dominate all urns. This creates a competition between local persistence and global interaction, which causes the phase transition we consider here. Such an interacting urn process provides a mathematical prototype for analyzing various weakly coupled systems in physics, economics, social sciences, etc.

There is substantial literature on reinforced processes and their analysis through stochastic approximation. The general dynamical-systems approach goes back to the works \cite{Pemantle1990,Benaim1999,Pemantle2007,Borkar2024}. Several interacting variants have been studied, including synchronization for linearly reinforced interacting Pólya and Friedman urns \cite{DaiPraLouisMinelli2014,CrimaldiDaiPraMinelli2016,Sahasrabudhe2016}, nonlinear and collective reinforcement \cite{CostaJordan2022,CrimaldiLouisMinelli2023,Mirebrahimi2024}, and broad network formulations addressing synchronization and asymptotic polarization \cite{AlettiCrimaldiGhiglietti2024Polarization,AlettiCrimaldiGhiglietti2024Asymptotics}. Graph-based and network urn models have also developed rapidly \cite{BenaimBenjaminiChenLima2015,vanDerHofstadHolmesKuznetsovRuszel2016,KaurSahasrabudhe2023,DahiyaSahasrabudhe2024,AlvesBarrosLima2025,KaurSahasrabudhe2025,PiresRosales2025,Melchionna2025}. The strongly reinforced interacting urns studied here were introduced in \cite{Launay2011}. Their generalized limiting case of exponential reinforcement was studied in \cite{LaunayLimic2012}. A more recent work \cite{Qin2024} developed the stochastic-approximation framework for power reinforcement, disproved the conjecture in \cite{Launay2011} that every positive interaction forces monopoly, and left open the sharp equality of the two critical parameters.

The primary purpose of this paper is to resolve the question whether the two critical parameters coincide. We use the
stochastic-approximation theorem in \cite{Qin2024} as a black-box input and reduce the problem to a deterministic
analysis of the limiting vector field. The central observation is that all off-diagonal equilibria can be encoded by a scalar analytic function \(\Phi_p\). Negative values of \(\Phi_p\) force the existence of a stable nondominating equilibrium, while a monotonicity identity in \(p\) shows that the set of parameters for which this happens is downward closed. This gives an explicit deterministic candidate for the critical point and proves that it is the unique phase-transition threshold.

Beyond identifying the sharp threshold, we ask how it depends on the reinforcement exponent. Increasing the exponent in power reinforcement makes the draw probability more responsive to deviations from an even color split. Therefore, one might expect that the interaction probability increases to overcome the opposing local preferences and that the critical threshold is nondecreasing in the exponent. We show that this intuition is false by giving a counterexample. This answers a question posed in \cite[Section~7]{Qin2024}.

\section{Setting and main result}

We consider two urns, each containing black and red balls. Let \(B_n(i)\) and \(R_n(i)\) denote the numbers of black and red balls in urn \(i\in\{1,2\}\) at time \(n\). Write
\begin{equation*}
 B_n^*=B_n(1)+B_n(2),
 \qquad
 R_n^*=R_n(1)+R_n(2),
\end{equation*}
and let
\begin{equation*}
 \calG_n=
 \sigma\{B_m(i),R_m(i):0\le m\le n,\ i=1,2\}.
\end{equation*}
We assume throughout that both colors are present in the system initially and, following \cite{Qin2024}, relabel the urns and colors if necessary so that
\begin{equation*}
 R_0(1)\ge 1,
 \qquad
 B_0(2)\ge 1.
\end{equation*}
Let \(\{W(n)\}_{n\ge0}\) be a reinforcement sequence with \(W(n)>0\) for \(n\ge1\) and \(W(0)\ge0\). Given \(p\in[0,1]\), the update at time \(n+1\) is defined as follows. Conditional on \(\calG_n\), the random variables \(\xi_{n+1}(1)\) and \(\xi_{n+1}(2)\) are independent Bernoulli random variables with
\begin{equation}\label{eq:urn-update-probability}
\PP\bigl(\xi_{n+1}(i)=1\mid \calG_n\bigr)
=
 p\frac{W(B_n^*)}{W(B_n^*)+W(R_n^*)}
 +(1-p)\frac{W(B_n(i))}{W(B_n(i))+W(R_n(i))}.
\end{equation}
Then
\begin{equation}\label{eq:urn-update}
 B_{n+1}(i)=B_n(i)+\xi_{n+1}(i),
 \qquad
 R_{n+1}(i)=R_n(i)+1-\xi_{n+1}(i).
\end{equation}
Thus, with probability \(p\), urn \(i\) draws according to the combined system, while with probability \(1-p\), it draws according to its own local composition.

\begin{figure}[t]
\centering
\begin{tikzpicture}[x=1cm,y=1cm,>=latex]

\tikzset{
  urnline/.style={line width=1pt, draw=brown!70!black},
  sepbar/.style={line width=1pt, dashed, draw=gray!70},
  blackball/.style={
    circle,
    draw=black,
    fill=white,
    line width=1pt,
    minimum size=6.1mm,
    inner sep=0pt
  },
  redball/.style={
    circle,
    draw=red!80!black,
    fill=white,
    line width=1pt,
    minimum size=6.1mm,
    inner sep=0pt
  }
}

\newcommand{\openurn}[3]{%
  \draw[urnline] (#1,#2+1.55) -- (#1,#2) -- (#1+#3,#2) -- (#1+#3,#2+1.55);
}

\begin{scope}[shift={(0,0)}]

  \node[font=\small] at (1.45,2.15) {Urn 1};
  \node[font=\small] at (4.55,2.15) {Urn 2};

  \openurn{0.35}{0}{2.2}
  \openurn{3.45}{0}{2.2}

  \node[blackball] at (1.05,0.90) {\scriptsize B};
  \node[blackball] at (1.38,0.35) {\scriptsize B};
  \node[blackball] at (2.20,0.98) {\scriptsize B};
  \node[redball]   at (0.75,0.35) {\scriptsize R};
  \node[blackball] at (2.15,0.35) {\scriptsize B};

  \node[redball]   at (4.20,0.86) {\scriptsize R};
  \node[redball]   at (4.60,0.35) {\scriptsize R};
  \node[redball]   at (5.3,1.00) {\scriptsize R};
  \node[blackball] at (3.80,0.35) {\scriptsize B};
  \node[blackball] at (5.30,0.35) {\scriptsize B};

  \draw[->, thin] (1.45,1.85) -- (1.45,1.28);
  \draw[->, thin] (4.55,1.85) -- (4.55,1.28);

  \node[align=center, font=\small] at (3.0,-0.55)
    {(a) Local step: with probability $1-p$,\\
     each urn draws from its own};
\end{scope}

\begin{scope}[shift={(7.6,0)}]

  \node[font=\small] at (2.45,2.15) {2 urns combined};

  \openurn{0.25}{0}{4.4}

  \draw[sepbar] (2.48,0.05) -- (2.48,1.5);

  \node[blackball] at (1.05,0.90) {\scriptsize B};
  \node[blackball] at (1.38,0.35) {\scriptsize B};
  \node[blackball] at (2.15,0.98) {\scriptsize B};
  \node[redball]   at (0.75,0.35) {\scriptsize R};
  \node[blackball] at (2.15,0.35) {\scriptsize B};

  \node[redball]   at (3.20,0.86) {\scriptsize R};
  \node[redball]   at (3.60,0.35) {\scriptsize R};
  \node[redball]   at (4.3,1.00) {\scriptsize R};
  \node[blackball] at (2.80,0.35) {\scriptsize B};
  \node[blackball] at (4.30,0.35) {\scriptsize B};

  \draw[->, thin] (2.45,1.9) -- (2.45,1.28);

  \node[align=center, font=\small] at (2.45,-0.55)
    {(b) Interaction step: with probability $p$,\\
     the draw is from the combined pool};
\end{scope}
\end{tikzpicture}
\caption{An illustration of the interacting urn mechanism for two urns.}
\label{fig:interacting-urn-mechanism}
\end{figure}
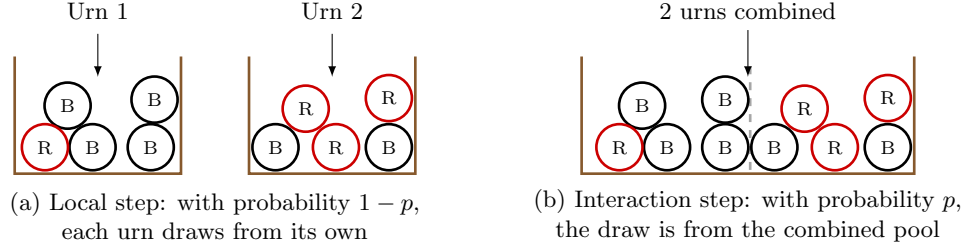

The proportions of black balls in the two urns are
\begin{equation}\label{eq:proportions}
 x_n=\frac{B_n(1)}{n+B_0(1)+R_0(1)},
 \qquad
 y_n=\frac{B_n(2)}{n+B_0(2)+R_0(2)}.
\end{equation}
We focus on power reinforcement
\begin{equation*}
 W(n)=n^\alpha,
 \qquad \text{for }\alpha>1,
\end{equation*}
and, when \(\alpha\) is an integer, on polynomial reinforcement of degree \(\alpha\). In both cases the associated stochastic approximation has the same limiting vector field.

The domination event is
\begin{equation}\label{eq:domination-event}
\calD=
\left\{\lim_{n\to\infty}(x_n,y_n)=(0,0)\right\}
\cup
\left\{\lim_{n\to\infty}(x_n,y_n)=(1,1)\right\}.
\end{equation}
Related to the domination event is the monopoly event, defined by
\begin{equation*}
\calM=
\{R_{n+1}^*=R_n^*\text{ eventually}\}
\cup
\{B_{n+1}^*=B_n^*\text{ eventually}\},
\end{equation*}
where eventually only one color is added to the whole system. For power and polynomial reinforcement with \(\alpha>1\), domination and monopoly coincide almost surely in the sense that \(\PP_p^{(\alpha)}(\calD\setminus\calM)=0\) (see \cite{Qin2024}).

For \(\alpha>1\), define
\begin{equation}\label{eq:critical-parameters}
 p_\alpha
 :=
 \inf\left\{q\in[0,1]:\PP_p^{(\alpha)}(\calD)=1
 \text{ for every }p\ge q\right\},
\end{equation}
and
\begin{equation}\label{eq:tilde-critical-parameter}
 \widetilde p_\alpha
 :=
 \sup\left\{q\in[0,1]:\PP_p^{(\alpha)}(\calD)<1
 \text{ for every }p\le q\right\}.
\end{equation}
In \cite{Qin2024}, it is proved that \(p_\alpha>0\), thus disproving the conjecture in \cite{Launay2011} that monopoly occurs for every \(p>0\) under power reinforcement. We prove the conjecture that \(\widetilde p_\alpha=p_\alpha\), proposed in \cite{Qin2024}, by identifying both parameters with a deterministic scalar critical point.

Fix \(\alpha>1\) and set
\begin{equation}\label{def:h}
 h(t)=h_\alpha(t):=\frac{t^\alpha}{t^\alpha+(1-t)^\alpha},
 \qquad t\in[0,1].
\end{equation}

For power reinforcement \(W(n)=n^\alpha\), the stochastic-approximation recursion for
\((x_n,y_n)\) has limiting drift \(F_p^{(\alpha)}\), where
\begin{equation}\label{eq:F}
\begin{cases}
F_{p,1}^{(\alpha)}(x,y)
=-x+(1-p)h(x)+p h\!\left(\dfrac{x+y}{2}\right),\\[2mm]
F_{p,2}^{(\alpha)}(x,y)
=-y+(1-p)h(y)+p h\!\left(\dfrac{x+y}{2}\right).
\end{cases}
\end{equation}
Thus, the associated deterministic system is
\begin{equation}\label{eq:ode}
 \dot z=F_p^{(\alpha)}(z), \qquad\text{for } z=(x,y)\in[0,1]^2.
\end{equation}

The same limiting vector field is obtained for polynomial reinforcements with leading term \(n^\alpha\). The lower-order terms contribute only an \(O(n^{-1})\) correction to the
conditional drift, which is included in the deterministic remainder. Hence, the limiting ODE is still \eqref{eq:ode}.
We shall use the following facts from \cite{Qin2024}.

\begin{proposition}[Stochastic approximation]\label{prop:Qin-input}
Let \(\alpha>1\) and \(p\in[0,1]\). Let \(\Lambda_p^{(\alpha)}\) be the equilibrium set of
\[
 \dot z=F_p^{(\alpha)}(z)
\]
on \([0,1]^2\). The vector field is a gradient field: $F_p^{(\alpha)}=\nabla L_p^{(\alpha)}$,
where
\begin{equation}\label{def:G}
 L_p^{(\alpha)}(x,y)
 =(1-p)\{G(x)+G(y)\}+2pG\!\left(\frac{x+y}{2}\right)
 -\frac{x^2+y^2}{2},
 \qquad
 G(t)=\int_0^t h(u)\dd u .
\end{equation}
Stability is understood relative to the invariant square \([0,1]^2\). An equilibrium
\(z_*\in\Lambda_p^{(\alpha)}\) is called stable if it is a strict local maximum of
\(L_p^{(\alpha)}\) relative to \([0,1]^2\). We call \(z_*\) unstable if the top
eigenvalue \(\lambda_+(z_*)\) of \(DF_p^{(\alpha)}(z_*)\) is positive.

Moreover:
\begin{enumerate}[label=\textup{(\roman*)}, leftmargin=*]
    \item for \(p>0\), the equilibrium set \(\Lambda_p^{(\alpha)}\) is finite;

    \item every equilibrium other than
    \[
      (0,0),\qquad \left(\frac12,\frac12\right),\qquad (1,1)
    \]
    lies in one of the two off-diagonal rectangles
    \[
      \left(0,\frac12\right)\times\left(\frac12,1\right),
      \qquad
      \left(\frac12,1\right)\times\left(0,\frac12\right);
    \]

    \item the urn process \((x_n,y_n)\) converges almost surely to an equilibrium of the ODE;

    \item if \(z_*\in\Lambda_p^{(\alpha)}\) is stable, then
    \[
      \PP_p^{(\alpha)}
      \left(\lim_{n\to\infty}(x_n,y_n)=z_*\right)>0;
    \]

    \item if \(z_*\in\Lambda_p^{(\alpha)}\) is unstable, then
    \[
      \PP_p^{(\alpha)}
      \left(\lim_{n\to\infty}(x_n,y_n)=z_*\right)=0 .
    \]
\end{enumerate}
\end{proposition}
\begin{figure}[t]
\centering
\begin{tikzpicture}[scale=5.4]
  \def\m{0.5}

  \fill[gray!14] (0,\m) rectangle (\m,1);
  \fill[gray!14] (\m,0) rectangle (1,\m);
  \fill[gray!35] (0,1) -- (\m,1) -- (\m,\m) -- cycle;
  \draw[thick] (0,0) rectangle (1,1);

  \draw[dashed] (\m,0) -- (\m,1);
  \draw[dashed] (0,\m) -- (1,\m);

  \draw[gray!60] (0,0) -- (1,1);

  \fill (0,0) circle (0.012);
  \fill (\m,\m) circle (0.012);
  \fill (1,1) circle (0.012);

  \node[below left] at (0,0) {$(0,0)$};
  \node[below right] at (\m,\m) {$\left(\frac12,\frac12\right)$};
  \node[above right] at (1,1) {$(1,1)$};
  \node at (0.75,0.25) {$\mathcal R_{-}$};
  \node at (0.25,0.75) {$\mathcal R_{+}$};
  \node at (0.4,0.8) {$\Omega_+$};

  \draw[->] (0,-0.015) -- (0,1.06);
  \draw[->] (-0.015,0) -- (1.06,0);

  \node[below] at (\m,0) {$\frac12$};

  \node[left] at (0,\m) {$\frac12$};

  \node[right] at (1.07,0) {$x$};
  \node[above] at (0,1.07) {$y$};
\end{tikzpicture}
\caption{The lightly shaded region is
\(\mathcal R_{+}\cup \mathcal R_{-}\).
By Proposition~\ref{prop:Qin-input}, every equilibrium other than
\((0,0)\), \(\left(\frac12,\frac12\right)\), and \((1,1)\)
must lie in the shaded region. The darker triangular region is
\[\Omega_+=\left\{(x,y):0<x<\frac12<y<1,\ x+y>1\right\},\] which is defined in \eqref{def:Dp} and used later in the scalar reduction.}
\label{fig:off-diagonal-rectangles}
\end{figure}
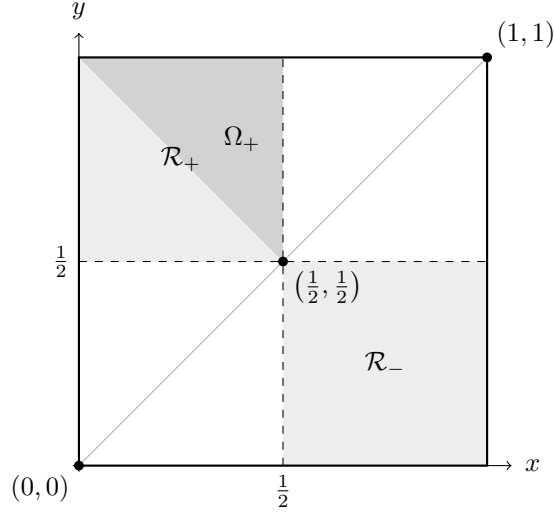
The deterministic critical point \(p_{\mathrm{det}}(\alpha)\) is defined in \eqref{eq:pdet-alpha} below after the scalar function \(\Phi_p\) has been constructed. Our main result is the following.

\begin{theorem}\label{thm:main}
For every \(\alpha>1\),
\begin{equation*}
 \widetilde p_\alpha=p_\alpha=p_{\mathrm{det}}(\alpha).
\end{equation*}
For polynomial reinforcements of integer degree \(m\ge 2\), the same conclusion holds with \(\alpha=m\), whenever the stochastic-approximation conclusions in Proposition~2.1 are available.
\end{theorem}

We also answer two questions posed in \cite[Section~7]{Qin2024} concerning the dependence of $p_\alpha$ on the reinforcement exponent. Proposition~\ref{prop:p-alpha-weak-limit} below determines the weak-reinforcement limit
\begin{equation*}
 \lim_{\alpha\downarrow1}\frac{p_\alpha}{\alpha-1}
 =\frac{1-s_0^2}{2s_0^2},
 \qquad
 2s_0\operatorname{arctanh}(s_0)=1.
\end{equation*}
Proposition~\ref{prop:p-alpha-not-monotone} shows that
$\alpha\mapsto p_\alpha$ is not nondecreasing by establishing the
explicit separation
\begin{equation*}
 p_{5/2}>\frac{31}{100}>p_3.
\end{equation*}

The rest of the paper proves Theorem~\ref{thm:main} and Propositions~\ref{prop:p-alpha-weak-limit} and \ref{prop:p-alpha-not-monotone}. The arguments are entirely deterministic once Proposition~\ref{prop:Qin-input} is granted.

\section{The anti-diagonal branch}

We begin with the equilibria on the anti-diagonal \(x+y=1\). On the anti-diagonal, the equilibrium equations reduce to a single scalar equation, and the stability of the resulting branch can be computed explicitly. The following lemma records this computation and introduces the first threshold needed in the proof.

\begin{lemma}[Anti-diagonal equilibria and stability]\label{lem:ad-stability}
Let \(\alpha>1\). The point \(\left(\frac12,\frac12\right)\) is an equilibrium for every
\(p\in[0,1]\). 
If $0\le p<\frac{\alpha-1}{\alpha}$,
then there is a unique \(u_{\alpha,p}\in[0,1/2)\) satisfying
\begin{equation}\label{eq:u-alpha-p}
 u=(1-p)h(u)+\frac p2 .
\end{equation}
The corresponding off-center anti-diagonal
equilibria are exactly
\[
 (u_{\alpha,p},1-u_{\alpha,p})\text{ and }
 (1-u_{\alpha,p},u_{\alpha,p}).
\]
If $ p\ge \frac{\alpha-1}{\alpha}$,
there is no off-center equilibrium on the anti-diagonal.

Furthermore, there exists a unique number
\[
 p_{\mathrm{ad},*}=p_{\mathrm{ad},*}(\alpha)
 \in\left(0,\frac{\alpha-1}{\alpha}\right)
\]
such that, for \(0\le p<(\alpha-1)/\alpha\),
\[
 \lambda_+(u_{\alpha,p},1-u_{\alpha,p})
 \begin{cases}
 <0, & \text{if }0\le p<p_{\mathrm{ad},*},\\
 =0, & \text{if }p=p_{\mathrm{ad},*},\\
 >0, & \text{if }p_{\mathrm{ad},*}<p<\dfrac{\alpha-1}{\alpha}.
 \end{cases}
\]
Consequently, the off-center anti-diagonal equilibria are asymptotically stable for
\(0\le p<p_{\mathrm{ad},*}\) and unstable for
\(p_{\mathrm{ad},*}<p<(\alpha-1)/\alpha\).
\end{lemma}
\begin{proof}
We first separate the central point. Since \(h(1/2)=1/2\), the point
\(\left(\frac12,\frac12\right)\) is an equilibrium of \eqref{eq:F} for every \(p\in[0,1]\).
It was shown in \cite{Qin2024} that 
$\lambda_+\!\left(\frac12,\frac12\right)=\alpha-1>0$, so therefore $(1/2,1/2)$ is always unstable. We now determine the remaining
equilibria on the anti-diagonal.

By symmetry, it is enough to consider an anti-diagonal equilibrium of
the form
\[
(x,y)=(x,1-x),
\qquad
\text{for }0<x<\frac12.
\]
Since by \eqref{def:h}, we have 
\[
h(1-x)=1-h(x),
\quad\text{and }
h\left(\frac{x+y}{2}\right)=h\left(\frac12\right)=\frac12,
\]
the two equilibrium equations 
\begin{equation}\label{e.equi_eqns}
    \begin{cases}
0 =-x+(1-p)h(x)+p h\!\left(\frac{x+y}{2}\right)\\
0 =-y+(1-p)h(y)+p h\!\left(\frac{x+y}{2}\right)
\end{cases}
\end{equation}
are equivalent. Using $y=1-x$, both equations reduce to
\begin{equation}\label{e.reduced_equi}
x=(1-p)h(x)+\frac p2.
\end{equation}
We now define
\begin{equation}\label{eq:p-ad-x}
\rho(x)
:=
\frac{1-2x}{1-2h(x)}
\quad
\text{and }p_{\mathrm{ad}}(x):=1-\rho(x).
\end{equation}
For a fixed $p\in\left(0,\frac{\alpha-1}{\alpha}\right)$, if $x\in(0,1)$ satisfies the reduced equilibrium equation \eqref{e.reduced_equi}, then the definitions in \eqref{eq:p-ad-x} imply that $\rho(x)=1-p$ and hence $p_\mathrm{ad}(x)=p$. Throughout this proof, we use the definitions in \eqref{eq:p-ad-x} conditioned on \eqref{e.reduced_equi} to hold.

We want to show that \(p_{\mathrm{ad}}\) is strictly increasing, that is to show that 
\begin{equation}\label{e.rho'}
\rho'(x)=2\frac{\left(1-2x\right)h'(x)-\left(1-2h(x)\right)}{\left(1-2h(x)\right)^2}
<0.
\end{equation}

We first directly compute from \eqref{def:h} that
\begin{equation}\label{eq.h'}
    h'(x)=\frac{\alpha x^{\alpha-1}(1-x)^{\alpha-1}}{\bigl(x^\alpha+(1-x)^\alpha\bigr)^2}.
\end{equation}

For \(0<x<\frac12\), taking the logarithmic differentiation of $h'(x)$ gives
\[
\frac{\partial}{\partial x}\log h'(x)=\frac{h''(x)}{h'(x)}=(\alpha-1)\left(\frac1x-\frac1{1-x}\right)+\frac{2\alpha\bigl((1-x)^{\alpha-1}-x^{\alpha-1}\bigr)}{x^\alpha+(1-x)^\alpha}.
\]
Both terms on the right-hand side of the last equation are positive and $h'(x)>0$ when
\(0<x<\frac12\). Then we have
\[
h''(x)>0,\qquad\text{for }0<x<\frac12,
\]
which shows that \(h\) is strictly convex in this interval. By the convexity of $h$, we get
\[
h\left(\frac12\right)>h(x)+h'(x)\left(\frac12-x\right).
\]
Plugging \(h\left(\frac12\right)=\frac12\) in the above display and rearranging the terms gives
\[
1-2h(x)>\left(1-2x\right)h'(x),
\]
proving \eqref{e.rho'}.

Moreover, we can compute that $\lim_{x\downarrow0}\rho(x)=1$ and $\lim_{x\uparrow\frac12}\rho(x)=\frac1\alpha$.
Hence,
$p_{\mathrm{ad}}:(0,\tfrac12)\rightarrow\left(0,\frac{\alpha-1}{\alpha}\right)$
is a strictly increasing bijection. In particular, for every
$p\in\left(0,\frac{\alpha-1}{\alpha}\right)$,
there is a unique \(x_p\in(0,\frac12)\) such that $p_{\mathrm{ad}}(x_p)=p$ and \((x_p,1-x_p)\) is an equilibrium with respect to the parameter $p$. This proves the part \eqref{eq:u-alpha-p} of this Lemma.

At an anti-diagonal equilibrium $(x,1-x)$ for $x\in\left(0,\frac12\right)$, the Jacobian matrix has two real eigenvalues:
\begin{equation}\label{e.ad.lambda_pm}
\begin{cases}
    \lambda_+(x,1-x)=-1+\alpha p+(1-p)h'(x),\\
    \lambda_-(x,1-x)=-1+(1-p)h'(x),
\end{cases}
\end{equation}
where $p$ satisfies $p=p_\mathrm{ad}(x)$.
Due to \eqref{eq.h'}, the function $h'$ is symmetric about $\frac{1}{2}$, so \eqref{e.ad.lambda_pm} implies that the reflected equilibrium $(1-x,x)$ has the same eigenvalues as $(x,1-x)$.

Since we have $\rho(x)=1-p_\mathrm{ad}(x)=1-p$ by \eqref{eq:p-ad-x}, we can rewrite $\lambda_+(x,1-x)$ as
\begin{equation}\label{e.ad-lambda_+.x_p}
    \lambda_+\left(x,1-x\right)=\alpha-1-\rho(x)\bigl(\alpha-h'(x)\bigr).
\end{equation}

Differentiating the above display gives
\begin{equation}\label{e.ad-lambda'_+.x}
\lambda'_+(x,1-x)
=
-\rho'(x)\bigl(\alpha-h'(x)\bigr)
+\rho(x)h''(x).
\end{equation}
Since \(h'\) is strictly increasing on \((0,\frac12)\),
\(h'(\frac12)=\alpha\), and $x\in\left(0,\frac12\right)$, we have
$\alpha-h'(x)>0$.
We have already shown that
\[
\rho'(x)<0,\qquad
\rho(x)>0,\qquad
\alpha-h'(x)>0,\qquad
h''(x)>0.
\]
Plugging these into \eqref{e.ad-lambda'_+.x} shows that
\begin{equation}\label{ineq.ad_lambda'>0}
\lambda'_+(x,1-x)>0.
\end{equation}

Finally, since \(h'(x)\to0\) as \(x\downarrow0\), we have
\begin{equation}
    \begin{cases}
        \lim_{x\downarrow0}\lambda_+(x,1-x)=-1,\\
        \lim_{x\uparrow\frac12}\lambda_+(x,1-x)=\alpha-1>0.
    \end{cases}
\end{equation}
Combining this with \eqref{ineq.ad_lambda'>0}, there exists a unique \(x_*\in(0,\frac12)\) such that
\[
\lambda_+(x_*,1-x_*)=0.
\]
We define
\begin{equation}\label{e.p_ad,*}
p_{\mathrm{ad},*}:=p_{\mathrm{ad}}(x_*).
\end{equation}
Since both \(p_{\mathrm{ad}}(\cdot)\) and \(\lambda_+(\cdot,1-\cdot)\) are strictly increasing and each anti-diagonal equilibrium $(x,1-x)$ satisfies \eqref{e.equi_eqns} for $p=p_\mathrm{ad}(x)$, we have
\[
\lambda_+(x,1-x)
\begin{cases}
<0,
&\text{if }0<p<p_{\mathrm{ad},*},\\[1mm]
=0,
&\text{if }p=p_{\mathrm{ad},*},\\[1mm]
>0,
&\text{if }p_{\mathrm{ad},*}<p<\dfrac{\alpha-1}{\alpha}.
\end{cases}
\]
By \cite{Qin2024}, the equilibrium is asymptotically stable when
\(\lambda_+<0\), and it is unstable when \(\lambda_+>0\). At
\(p=p_{\mathrm{ad},*}\), the equilibrium is nonhyperbolic and
linearization alone does not determine its nonlinear stability.

At \(p=0\), it can be extended to the endpoint equilibria
\((0,1)\) and \((1,0)\). Their corresponding Jacobian matrix is $-I$, so their top eigenvalue is negative, meaning that these equilibria are asymptotically stable.
\end{proof}

\section{The scalar function \texorpdfstring{$\Phi_p$}{Phi}}

We now encode all equilibria in one off-diagonal sector by a single scalar equation. The first elementary fact will be used repeatedly.

\begin{lemma}\label{lem:hprime-monotone}
The function \(h'\) is symmetric about \(1/2\) and strictly decreases as \(|t-1/2|\) increases. In particular,
\begin{equation*}
 h''(t)<0,
 \qquad\text{for all } t\in\left(\frac12,1\right).
\end{equation*}
\end{lemma}

\begin{proof}
Formula \eqref{eq.h'} gives the symmetry. Let \(s=\log(t/(1-t))\). Then \(t>1/2\) is equivalent to \(s>0\), and
\begin{equation*}
 h'(t(s))=
 \alpha\frac{\cosh^2(s/2)}{\cosh^2(\alpha s/2)}.
\end{equation*}
Thus, we obtain
\begin{equation*}
 \frac{\dd}{\dd s}\log h'(t(s))
 =\tanh\left(\frac s2\right)-\alpha\tanh\left(\frac{\alpha s}{2}\right)<0,
 \qquad \text{for }s>0 \text{ and }\alpha>1.
\end{equation*}
Since \(s\mapsto t(s)\) is strictly increasing, \(h'\) is strictly decreasing on \((1/2,1)\). The statement follows by symmetry.
\end{proof}

Throughout this section, we assume $p\in\left(0,\frac{\alpha-1}{\alpha}\right)$.
Define
\begin{equation}\label{eq:Ftilde}
 \widetilde F_p(t,z):=-t+(1-p)h(t)+ph(z).
\end{equation}
If \(z=(x+y)/2\), then the equilibrium equations are
\begin{equation*}
 \widetilde F_p(x,z)=0,
 \qquad
 \widetilde F_p(y,z)=0.
\end{equation*}
Thus, when \(z\) is fixed, both coordinates of an equilibrium must be roots of the same scalar equation.

We focus on the sector
\begin{equation}\label{def:Omega-plus}
 \Omega_+:=\left\{(x,y):0<x<\frac12<y<1,\ x+y>1\right\}.
\end{equation}
For \((x,y)\in\Omega_+\), we have \(z=(x+y)/2\in(1/2,3/4)\) and \(y>z>1/2>x\). Thus the upper coordinate of an equilibrium in \(\Omega_+\) must be a root of \(\widetilde F_p(t,z)=0\) with \(t>1/2\). The next lemma proves that this root is unique and depends analytically on \(z\).

\begin{lemma}\label{lem:upper-branch}
For each \(z\in(1/2,3/4)\), the equation
\begin{equation*}
 \widetilde F_p(t,z)=0
\end{equation*}
has a unique solution \(q_p(z)\) in \((1/2,1)\). This solution lies on the decreasing part of the graph \(t\mapsto\widetilde F_p(t,z)\), and this solution satisfies \(q_p(z)>z\). Moreover,
\begin{equation}\label{def:Dp}
 D_p(z):=\frac{\partial}{\partial t}\widetilde F_p(q_p(z),z)
 =-1+(1-p)h'(q_p(z))<0.
\end{equation}
The function \(q_p\) is analytic on \((1/2,1)\) and extends analytically to a neighborhood of \(z=1/2\).
\end{lemma}

\begin{proof}
Let
\begin{equation}\label{def:g}
 g(t):=-t+(1-p)h(t)+\frac p2.
\end{equation}
Then we can write
\begin{equation*}
 \widetilde F_p(t,z)=g(t)+p\left(h(z)-\frac12\right).
\end{equation*}
By Lemma~\ref{lem:hprime-monotone} and the fact that \(g'(t)=-1+(1-p)h'(t)\), there is \(t_1\in(1/2,1)\) such that \(g\) increases on \([1/2,t_1]\) and decreases on \([t_1,1]\). Also, we have
\begin{equation*}
 g(t_1)>0,
 \qquad
 g(1)=-\frac p2.
\end{equation*}
Recall that $h'>0$ and $h\left(1/2\right)=1/2$. Hence, we can choose \(\eps>0\) small enough so that for every \(z\in(1/2-\eps,1)\), we have $\widetilde F_p(1,z)<0<\widetilde F_p(t_1,z)$, or equivalently, we have
\begin{equation*}
 -p\left(h(z)-\frac12\right)\in(g(1),g(t_1)).
\end{equation*}
From above, since \(g\) is strictly decreasing on \((t_1,1)\), for every \(z\in(1/2-\eps,1)\), there is a unique \(q_p(z)\in(t_1,1)\) satisfying
\begin{equation*}
 g(q_p(z))=-p\left(h(z)-\frac12\right),
\end{equation*}
or equivalently \(\widetilde F_p(q_p(z),z)=0\). Due to \(g'(t)<0\) on \((t_1,1)\), the inverse function $g^{-1}$ is analytic there by the implicit function theorem, and we can write
\begin{equation}\label{def:qp}
 q_p(z)=g^{-1}\left(-p\left(h(z)-\frac12\right)\right),
\end{equation}
which is analytic on \((1/2-\eps,1)\).

It remains to show \(q_p(z)>z\) for \(z>1/2\). If \(z\le t_1\), this follows from \(q_p(z)>t_1\). If \(z>t_1\), then
\begin{equation}\label{ineq:h-greater-than-id}
 \widetilde F_p(z,z)=h(z)-z>0,
\end{equation}
where the last inequality follows from the definition of $h$ in \eqref{def:h} that \(h(z)>z\) for \(z>1/2\). Since \(t\mapsto\widetilde F_p(t,z)\) strictly decreases on \((t_1,1)\) and vanishes at \(q_p(z)\), we get \(q_p(z)>z\). Finally, comparing \eqref{def:g} with the definition of $D_p(z)$ in \eqref{def:Dp}, we see that \(D_p(z)=g'(q_p(z))<0\), proving the last inequality.
\end{proof}

Define
\begin{equation}\label{eq:xp}
 x_p(z):=2z-q_p(z),
\end{equation}
and
\begin{equation}\label{eq:Phi}
 \Phi_p(z):=\widetilde F_p(x_p(z),z)
 =\widetilde F_p(2z-q_p(z),z).
\end{equation}
More explicitly, we write
\begin{equation}\label{eq:explicit_Phi}
 \Phi_p(z)=-x_p(z)+(1-p)h(x_p(z))+ph(z).
\end{equation}
Since \(q_p(3/4)\in(3/4,1)\), it follows that
\begin{equation*}
 x_p\!\left(\frac34\right)=\frac32-q_p\!\left(\frac34\right)
 \in\left(\frac12,\frac34\right).
\end{equation*}
Using the argument that derived \eqref{ineq:h-greater-than-id}, we have \(h(x_p(3/4))>x_p(3/4)\) and \(h(3/4)>3/4\). Hence, we get
\begin{equation*}
\Phi_p\!\left(\frac34\right)
>
-x_p\!\left(\frac34\right)+(1-p)x_p\!\left(\frac34\right)+p\frac34
=p\left(\frac34-x_p\!\left(\frac34\right)\right)>0.
\end{equation*}
This proves the following simple lemma.
\begin{lemma}\label{lem:endpoint-sign}
For every \(0<p<({\alpha-1})/{\alpha}\), we have
$\Phi_p\!\left(\frac34\right)>0$.
\end{lemma}

By Lemma~\ref{lem:upper-branch}, the functions \(q_p\), \(x_p\), and \(\Phi_p\) are analytic near \(z=1/2\). Denote by \(u_{\alpha,p}\) the unique solution in \((0,1/2)\) of
\begin{equation*}
 u=(1-p)h(u)+\frac p2.
\end{equation*}
Then using \eqref{def:qp} and \eqref{eq:u-alpha-p}, we see that
\begin{equation*}
 q_p\!\left(\frac12\right)=1-u_{\alpha,p},
 \qquad
 x_p\!\left(\frac12\right)=u_{\alpha,p},
 \qquad
 \Phi_p\!\left(\frac12\right)=0.
\end{equation*}
The zero \(\Phi_p\!\left(\frac12\right)=0\) corresponds to the anti-diagonal equilibrium. In the next lemma, we identify the zeros of $\Phi_p$ in \((1/2,3/4)\).

\begin{lemma}\label{lem:encoding}
The zeros of \(\Phi_p\) in \((1/2,3/4)\) are in one-to-one correspondence with the equilibria in \(\Omega_+\). The correspondence is
\begin{equation*}
 z\mapsto (x,y)=(2z-q_p(z),q_p(z)).
\end{equation*}
\end{lemma}

\begin{proof}
Let \((x,y)\in\Omega_+\) be an equilibrium and put \(z=(x+y)/2\). Then, we have \(z\in(1/2,3/4)\) and
\begin{equation*}
 \widetilde F_p(x,z)=0,
 \qquad
 \widetilde F_p(y,z)=0.
\end{equation*}
Since \(y>z>1/2\), Lemma~\ref{lem:upper-branch} gives \(y=q_p(z)\). Hence, we obtain \(x=2z-q_p(z)=x_p(z)\) and \(\Phi_p(z)=0\).

Conversely, we suppose \(z\in(1/2,3/4)\) and \(\Phi_p(z)=0\). Set
\begin{equation*}
 x=x_p(z)=2z-q_p(z)\text{ and }
 y=q_p(z).
\end{equation*}
The definitions of \(q_p\) and \(\Phi_p\) give \(\widetilde F_p(x,z)=\widetilde F_p(y,z)=0\), so \((x,y)\) is an equilibrium. Since \(q_p(z)>z\), we have \(0<x<z<y<1\). If \(x\ge1/2\), then \(h(x)\ge x\), \(h(z)>z\), and \(z>x\), which gives
\begin{equation}\label{ineq.bd_Phi_p(z)>0}
\Phi_p(z)
 =-x+(1-p)h(x)+ph(z)
 >-x+(1-p)x+pz=p(z-x)>0,
\end{equation}
which contradicts $\Phi_p(z)=0$. Therefore, we have \(0<x<1/2<y<1\). Since \(x+y=2z>1\), we have \((x,y)\in\Omega_+\).
\end{proof}
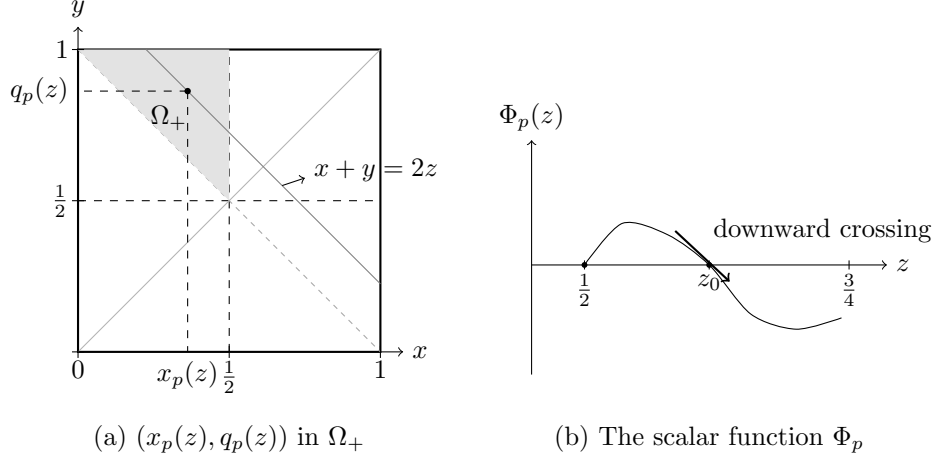
\begin{figure}[t]
\centering
\begin{tikzpicture}[x=1cm,y=1cm]

\begin{scope}[shift={(0,0)}]
  \def\m{0.5}

  \draw[thick] (0,0) rectangle (4,4);

  \draw[dashed] (2,0) -- (2,4);
  \draw[dashed] (0,2) -- (4,2);

  \draw[gray!60] (0,0) -- (4,4);
  \draw[dash pattern=on 2pt off 2pt, gray!70] (0,4) -- (4,0);

  \fill[gray!22] (0,4) -- (2,4) -- (2,2) -- cycle;

  \draw[gray] (0.9,4) -- (4,0.9);
  \draw[->] (2.7,2.2) -- (3,2.3);
  \node[right] at (3,2.4) {$x+y=2z$};
  \fill (1.45,3.45) circle (1.2pt);
  
  \draw[dashed] (1.45,3.45) -- (1.45,0);
  \draw[dashed] (1.45,3.45) -- (0,3.45);

  \node at (1.2,3.1) {$\Omega_+$};

  \node[below] at (1.45,0) {$x_p(z)$};
  \node[left] at (0,3.45) {$q_p(z)$};

  \draw (0,-0.08) -- (0,0.08);
  \draw (2,-0.08) -- (2,0.08);
  \draw (4,-0.08) -- (4,0.08);
  \draw[->](0,4) -- (0,4.3);
  \draw[->](4,0) -- (4.3,0);
  \draw (-0.08,0) -- (0.08,0);
  \draw (-0.08,2) -- (0.08,2);
  \draw (-0.08,4) -- (0.08,4);

  \node[below] at (0,0) {$0$};
  \node[below] at (2,0) {$\frac12$};
  \node[below] at (4,0) {$1$};

  \node[left] at (0,2) {$\frac12$};
  \node[left] at (0,4) {$1$};

  \node[right] at (4.3,0) {$x$};
  \node[above] at (0,4.3) {$y$};

  \node at (2,-1.15) {(a) $(x_p(z),q_p(z))$ in \(\Omega_+\)};
\end{scope}
\begin{scope}[shift={(6.0,1.15)}]

  \draw[->] (0,0) -- (4.7,0) node[right] {$z$};
  \draw[->] (0,-1.45) -- (0,1.65) node[above] {$\Phi_p(z)$};
  \draw (0.7,-0.06) -- (0.7,0.06);
  \draw (2.35,-0.06) -- (2.35,0.06);
  \draw (4.2,-0.06) -- (4.2,0.06);

  \node[below] at (0.7,0) {$\frac12$};
  \node[below] at (2.35,0) {$z_0$};
  \node[below] at (4.2,0) {$\frac34$};

  \draw[thin]
    plot[smooth] coordinates {
      (0.7,0)
      (1.2,0.55)
      (1.8,0.40)
      (2.35,0)
      (2.9,-0.65)
      (3.5,-0.85)
      (4.1,-0.70)
    };

  \fill (0.7,0) circle (1.1pt);
  \fill (2.35,0) circle (1.1pt);

  \draw[->, thick] (1.9,0.45) -- (2.62,-0.22);
  \node[above] at (3.85,0.18) {downward crossing};

  \node at (2.35,-2.3) {(b) The scalar function \(\Phi_p\)};
\end{scope}

\end{tikzpicture}
\caption{A point in \(\Omega_+\) lying on the line \(x+y=2z\) is an equilibrium precisely when
$\Phi_p(z)=\widetilde F_p(x_p(z),z)=0$.
The zeros of \(\Phi_p\) encode the equilibria in \(\Omega_+\), and a downward crossing corresponds to a stable nondominating equilibrium.}
\label{fig:scalar-reduction}
\end{figure}
The following lemma is a simple fact on the derivative at the anti-diagonal point of $\Phi_p$ (i.e. $\Phi'_p\left(\frac12\right)$).
\begin{lemma}\label{lem:Phi-half}
For \(u_{\alpha,p}\) defined by \eqref{eq:u-alpha-p}, one has
\begin{equation*}
 \Phi_p'\!\left(\frac12\right)
 =2\lambda_+(u_{\alpha,p},1-u_{\alpha,p}).
\end{equation*}
\end{lemma}
\begin{proof}
Recall that $(u_{\alpha,p},1-u_{\alpha,p})$ is the anti-diagonal equilibrium on the boundary of $\Omega_+$ and that $q_p\left(\frac{1}{2}\right)=1-u_{\alpha,p}$. Let
\begin{equation*}
 d:=-1+(1-p)h'(1-u_{\alpha,p})=-1+(1-p)h'(u_{\alpha,p}),
\end{equation*}
where the last equality is due to the symmetry $h'(1-u_{\alpha,p})=h'(u_{\alpha,p})$.

We use \eqref{def:Dp} to see that
\begin{equation}\label{e.partial_tF.1/2=d}
 \frac{\partial}{\partial t}\widetilde F_p
 \left(q_p\left(\frac{1}{2}\right),\frac{1}{2}\right)= \frac{\partial}{\partial t}\widetilde F_p
 \left(1-q_p\left(\frac{1}{2}\right),\frac{1}{2}\right)=d.
\end{equation}
According to the computation in \cite{Qin2024}, the top eigenvalue of $DF^{(\alpha)}_p(u_{\alpha,p},1-u_{\alpha,p})$ is
\begin{equation}\label{eq:lambda-plus-antidiagonal}
    \lambda_+(u_{\alpha,p},1-u_{\alpha,p})=-1+p h'\left(\frac{1}{2}\right)+(1-p)h'(u_{\alpha,p})=d+p\alpha,
\end{equation}
where in the last equality we use the fact that $h'\left(\frac{1}{2}\right)=\alpha$.
We now compute the derivative of $\Phi_p$. The identity
\begin{equation*}
 \widetilde F_p(q_p(z),z)=0
\end{equation*}
holds for the analytic extension of $q_p$ near $z=\frac{1}{2}$.
Differentiating this identity gives
\begin{equation*}
 \frac{\partial}{\partial t}\widetilde F_p(q_p(z),z)q_p'(z)+ph'(z)=0.
\end{equation*}
Evaluating it at $z=\frac{1}{2}$ using \eqref{e.partial_tF.1/2=d} yields
\begin{equation}\label{eq:q-prime-half}
 d q_p'\left(\frac{1}{2}\right)+p\alpha=0.
\end{equation}

Differentiating $\Phi_p(z)$ gives
\begin{align*}
 \Phi_p'(z)
 &=
 \frac{\partial }{\partial t}\widetilde F_p(2z-q_p(z),z)
 \left(2-q_p'(z)\right)
 +
 ph'(z).
\end{align*}
At $z=\frac{1}{2}$, we similarly use \eqref{e.partial_tF.1/2=d} to get
\begin{equation*}
 \Phi_p'\left(\frac{1}{2}\right)
 =
 d\left(2-q_p'\left(\frac{1}{2}\right)\right)
 +
 p\alpha.
\end{equation*}
By \eqref{eq:q-prime-half} and \eqref{eq:lambda-plus-antidiagonal}, the above display implies that
\begin{equation*}
 \Phi_p'\left(\frac{1}{2}\right)=2d+2p\alpha=2\lambda_+(u_{\alpha,p},1-u_{\alpha,p}),
\end{equation*}
proving the result.
\end{proof}

\section{Stability in scalar form}

Let \((x,y)=(2z-q_p(z),q_p(z))\), where \(z=(x+y)/2\). The Jacobian of \(F_p^{(\alpha)}\) at \((x,y)\) is
\begin{equation}\label{eq:J-ABD}
 J_{(x,y)}=
 \begin{pmatrix}
 A(x)+B(z) & B(z)\\
 B(z) & D(y)+B(z)
 \end{pmatrix},
\end{equation}
where we set
\begin{equation}\label{eq:ABD}
 A(x):=-1+(1-p)h'(x),
 \qquad
 D(y):=-1+(1-p)h'(y),
 \qquad
 B(z):=\frac p2 h'(z).
\end{equation}
Differentiating the equation \(\widetilde F_p(q_p(z),z)=0\) gives $D(y)q_p'(z)+2B(z)=0$, which gives
\begin{equation}\label{eq:q-prime}
 q_p'(z)=-\frac{2B(z)}{D(y)}.
\end{equation}
By \eqref{def:Dp}, we see that \(D(y)<0\), and since \(B(z)>0\), we have \(q_p'(z)>0\). On the other hand, differentiating \(\Phi_p\) gives
\begin{equation}\label{eq:Phi-prime}
 \Phi_p'(z)=A(x)(2-q_p'(z))+2B(z).
\end{equation}
Combining \eqref{eq:q-prime} and \eqref{eq:Phi-prime}, we obtain the determinant identity
\begin{align}
 \frac{D(y)}{2}\Phi_p'(z)
 &=\frac{D(y)}{2}\left[A(x)\left(2+\frac{2B(z)}{D(y)}\right)+2B(z)\right]\nonumber\\
 &=A(x)D(y)+A(x)B(z)+B(z)D(y)
 =\det J_{(x,y)}.\label{eq:det-identity}
\end{align}
By Lemma~\ref{lem:encoding}, there is a one-to-one correspondence between the zeros of $\Phi_p$ on $(1/2,3/4)$ and the equilibria in $\Omega_+$.
We now prove the following lemma that provides sufficient conditions for the equilibria in $\Omega_+$ to be asymptotically stable. 
\begin{lemma}\label{lem:nondegenerate-stable}
Suppose \(z_0\in(1/2,3/4)\) satisfies
\begin{equation*}
 \Phi_p(z_0)=0,
 \qquad
 \Phi_p'(z_0)<0,
 \qquad
 0<q_p'(z_0)<2.
\end{equation*}
Then the corresponding equilibrium
\begin{equation*}
 (x_0,y_0)=(2z_0-q_p(z_0),q_p(z_0))
\end{equation*}
is asymptotically stable. Or equivalently, one has \(\lambda_+(x_0,y_0)<0\).
\end{lemma}

\begin{proof}
Fix \(z=z_0\in(1/2,3/4)\), write \(q=q_p'(z_0)\), \(B=B(z_0)\), \(A=A(x_0)\), and \(D=D(y_0)\). By \eqref{eq:det-identity}, since \(D<0\) and \(\Phi_p'(z_0)<0\), we have \(\det J_{(x_0,y_0)}>0\). It remains to prove \(\tr J_{(x_0,y_0)}<0\). From \eqref{eq:q-prime}, we already know that
\begin{equation*}
 D=-\frac{2B}{q}.
\end{equation*}
Moreover, using \eqref{eq:Phi-prime} and \(\Phi_p'(z_0)<0\) gives
\begin{equation*}
 A(2-q)+2B<0.
\end{equation*}
Since \(0<q<2\), one has
\begin{equation*}
 A<-\frac{2B}{2-q}.
\end{equation*}
Therefore, we can compute
\begin{equation*}
 \tr J_{(x_0,y_0)}=A+D+2B
 <-\frac{2B}{2-q}-\frac{2B}{q}+2B.
\end{equation*}
For \(0<q<2\), it is clear that
\begin{equation*}
 \frac1q+\frac1{2-q}\ge 2.
\end{equation*}
Since \(B>0\), the last two displays imply \(\tr J_{(x_0,y_0)}<0\). The matrix is real symmetric, so positive determinant and negative trace imply that both eigenvalues are negative.
\end{proof}

\section{Monotonicity in the interaction parameter}

The next lemma says that as long as the lower coordinate remains below \(1/2\), the monotonicity of the scalar function \(\Phi_p\) is increasing in the interaction parameter $p$. This is an important monotonicity property to be used in the proofs later.

\begin{lemma}\label{lem:Phi-monotone}
Fix \(z\in(1/2,3/4)\). Assume that
\begin{equation*}
 x_p(z)=2z-q_p(z)\in\left(0,\frac12\right).
\end{equation*}
Then one has
\begin{equation*}
 \frac{\partial}{\partial p}\Phi_p(z)>0.
\end{equation*}
\end{lemma}

\begin{proof}
The identity \(\widetilde F_p(q_p(z),z)=0\) is equivalent to
\begin{equation}\label{eq:q-equation-p}
 q_p(z)=(1-p)h(q_p(z))+ph(z).
\end{equation}
Since \(\partial_t\widetilde F_p(q_p(z),z)<0\), the implicit function theorem allows us to differentiate \(q_p(z)\) with respect to \(p\). From \eqref{eq:q-equation-p}, we obtain
\begin{equation}\label{eq:partial-p-q}
 \frac{\partial}{\partial p}q_p(z)
 =\frac{h(z)-h(q_p(z))}{1-(1-p)h'(q_p(z))}.
\end{equation}
By \eqref{def:Dp} from Lemma~\ref{lem:upper-branch}, the denominator in \eqref{eq:partial-p-q} is positive. By the facts that $h$ is increasing and $q_p(z)>z$ from Lemma~\ref{lem:upper-branch}, the numerator in \eqref{eq:partial-p-q} is negative. Hence, we deduce that
\begin{equation}\label{eq:pq-negative}
 \frac{\partial}{\partial p}q_p(z)<0.
\end{equation}
Recall $\Phi_p$ defined in \eqref{eq:explicit_Phi}. Since \(x_p(z)=2z-q_p(z)\), differentiating $\Phi_p(z)$ in \(p\) gives
\begin{equation}\label{eq:partial-p-Phi}
 \frac{\partial}{\partial p}\Phi_p(z)
 =h(z)-h(x_p(z))+\bigl(1-(1-p)h'(x_p(z))\bigr)\frac{\partial}{\partial p}q_p(z).
\end{equation}
Combining \eqref{eq:partial-p-Phi} with \eqref{eq:partial-p-q} and rearranging the terms, we get
\begin{align}
&\bigl(1-(1-p)h'(q_p(z))\bigr)\frac{\partial}{\partial p}\Phi_p(z)\nonumber\\
&=\bigl(1-(1-p)h'(q_p(z))\bigr)\bigl(h(z)-h(x_p(z))\bigr)\nonumber\\
&\qquad -\bigl(1-(1-p)h'(x_p(z))\bigr)\bigl(h(q_p(z))-h(z)\bigr).\label{eq:key-monotonicity-identity}
\end{align}
Let \(\ell=q_p(z)-z>0\). Then we can rewrite \(x_p(z)=z-\ell\), \(q_p(z)=z+\ell\). For every \(s\in(0,\ell]\), since $z>1/2$, we must have
\begin{equation*}
 \left|z-s-\frac12\right|<\left|z+s-\frac12\right|.
\end{equation*}
By Lemma~\ref{lem:hprime-monotone}, the above display implies that
\begin{equation*}
 h'(z-s)>h'(z+s)>0,
 \qquad \text{for all } 0<s\le\ell.
\end{equation*}
Integrating them in \(s\) gives
\begin{equation}\label{eq:h-increment}
 h(z)-h(x_p(z))>h(q_p(z))-h(z)>0,
\end{equation}
and also taking $s=l$ shows
\begin{equation}\label{eq:hprime-comparison}
 h'(x_p(z))>h'(q_p(z)).
\end{equation}
Set
\begin{equation*}
 C_x=1-(1-p)h'(x_p(z)),
 \qquad
 C_q=1-(1-p)h'(q_p(z)).
\end{equation*}
Clearly, we have $C_q>0$ due to \eqref{def:Dp}, and we have \(C_x<C_q\) due to \eqref{eq:hprime-comparison}. Now we consider two cases. If \(C_x\le0\), then \eqref{eq:key-monotonicity-identity} is strictly positive since it can be written as $C_q(h(z)-h(x_p(z)))-C_x(h(q_p(z))-h(z))$ and \eqref{eq:h-increment} can be applied. If \(C_x>0\), then \eqref{eq:h-increment} and \(C_x<C_q\) again make the right-hand side strictly positive. Since \(C_q>0\), the desired inequality follows.
\end{proof}

\section{No-source and downward crossings}
Recall the definition of the domination event $\mathcal D$ in \eqref{eq:domination-event}. Based on this, we say that an equilibrium is \emph{nondominating} if it is neither $(0,0)$ nor $(1,1)$, since $(0,0)$ and $(1,1)$ are the dominating equilibria.

In this section, the goal is to show that a zero of \(\Phi_p\) that is of a certain type corresponds to a stable nondominating equilibrium. 
The first proposition proves that a crossing zero of $\Phi_p$ gives a stable equilibrium when the slope at this zero is negative.

\begin{proposition}\label{prop:no-source}
Let \(z_0\in(1/2,3/4)\) be a zero of \(\Phi_p\), and set
\begin{equation*}
 y_0=q_p(z_0),
 \qquad
 x_0=2z_0-q_p(z_0).
\end{equation*}
Assume \((x_0,y_0)\in\Omega_+\). Then \(q_p'(z_0)\in(0,2)\). Consequently, if \(\Phi_p'(z_0)<0\), then \((x_0,y_0)\) is asymptotically stable and \(\lambda_+(x_0,y_0)<0\).
\end{proposition}

\begin{proof}
Fix \(x_0\) and define
\begin{equation*}
 \psi(y):=F_{p,2}^{(\alpha)}(x_0,y)
 =-y+(1-p)h(y)+p h\!\left(\frac{x_0+y}{2}\right),
 \qquad\text{for } y\in[1-x_0,1].
\end{equation*}
Since \(0<x_0<1/2\), both \(y\) and \((x_0+y)/2\) lie in \([1/2,1]\) for \(y\in[1-x_0,1]\). Then we can apply Lemma~\ref{lem:hprime-monotone} to see that
\begin{equation}\label{eq:concavity-psi}
 \psi''(y)=(1-p)h''(y)+\frac p4 h''\!\left(\frac{x_0+y}{2}\right)<0,
 \qquad\text{for } y\in(1-x_0,1),
\end{equation}
so \(\psi\) is strictly concave on $(1-x_0,1)$. Also, plugging in $y=1$ gives
\begin{equation*}
 \psi(1)=p\left[h\!\left(\frac{x_0+1}{2}\right)-1\right]<0.
\end{equation*}
At \(y=1-x_0\), using \(h(1-x_0)=1-h(x_0)\), we get
\begin{equation}\label{eq:psi-one-minus-x}
 \psi(1-x_0)=x_0-(1-p)h(x_0)-\frac p2
 =-F_{p,1}^{(\alpha)}(x_0,1-x_0).
\end{equation}
Since \((x_0,y_0)\) is an equilibrium and \(y_0>1-x_0\), while \(F_{p,1}^{(\alpha)}(x_0,y)\) is strictly increasing in \(y\), we must have
\begin{equation*}
 F_{p,1}^{(\alpha)}(x_0,1-x_0)<F_{p,1}^{(\alpha)}(x_0,y_0)=0.
\end{equation*}
Combining this with \eqref{eq:psi-one-minus-x}, we obtain \(\psi(1-x_0)>0\).

We claim \(\psi'(y_0)<0\). Suppose otherwise, then the strict concavity of $\psi$ implies \(\psi'(y)\ge0\) for \(y\in[1-x_0,y_0]\). Hence, we arrive at \(\psi(1-x_0)\le\psi(y_0)=0\), a contradiction. Therefore, we have
\begin{equation}\label{eq:J22-negative}
 \psi'(y_0)=-1+(1-p)h'(y_0)+\frac p2 h'(z_0)\stackrel{\eqref{eq:ABD}}{=}D(y_0)+B(z_0)<0.
\end{equation}
Thus, we get \(D(y_0)<-B(z_0)<0\). Applying this to \eqref{eq:q-prime} gives us the range
\begin{equation*}
 q_p'(z_0)=-\frac{2B(z_0)}{D(y_0)}\in(0,2).
\end{equation*}
The final assertion follows from Lemma~\ref{lem:nondegenerate-stable}.
\end{proof}
Let $z_0\in[1/2,3/4)$ be a zero of $\Phi_p$. We say that it is a \emph{downward crossing} of $\Phi_p$ if there exists \(\eps>0\) such that
\begin{equation}\label{eq:downward-crossing-assumption}
 (z-z_0)\Phi_p(z)<0
\end{equation}
for every \(z\in[1/2,3/4)\) with \(0<|z-z_0|<\eps\). See Figure~\ref{fig:scalar-reduction} for an illustration.
\begin{lemma}\label{lem:downward-crossing-stable}
Assume \(0<p<({\alpha-1})/{\alpha}\). Let \(z_0\in[1/2,3/4)\) be a zero of \(\Phi_p\), where at \(z_0=1/2\) the functions \(q_p\), \(x_p\), and \(\Phi_p\) are understood through their analytic extensions. Set
\begin{equation*}
 x_0=x_p(z_0),
 \qquad
 y_0=q_p(z_0),
\end{equation*}
and assume \(0<x_0<1/2<y_0<1\). Suppose that $z_0$ is a downward crossing of $\Phi_p$. Then \((x_0,y_0)\) is a strict local maximum of \(L_p^{(\alpha)}\). In particular, it is an asymptotically stable equilibrium.
\end{lemma}

\begin{proof}
We split the proof into two cases.

\smallskip
\noindent\textbf{Case 1: \(z_0\in(1/2,3/4)\).}
Since \(x_0+y_0=2z_0>1\), the assumptions imply \((x_0,y_0)\in\Omega_+\). By Proposition~\ref{prop:no-source}, we know that \(q_p'(z_0)<2\), and hence
\begin{equation}\label{eq:xp-prime-positive}
 x_p'(z_0)=2-q_p'(z_0)>0.
\end{equation}
Thus, the map \(z\mapsto x_p(z)\) is locally strictly increasing near \(z_0\). By \eqref{eq:J22-negative} from the proof of Proposition~\ref{prop:no-source}, we get
\begin{equation}\label{eq:J22-neg-for-IFT}
 \frac{\partial F_{p,2}^{(\alpha)}}{\partial y}(x_0,y_0)<0.
\end{equation}
Since \(F_p^{(\alpha)}=\nabla L_p^{(\alpha)}\), this is
\begin{equation}\label{eq:Lyy-negative}
 \frac{\partial^2 L_p^{(\alpha)}}{\partial y^2}(x_0,y_0)<0.
\end{equation}
By continuity, there are a constant \(c>0\) and a neighborhood \(U\) of \((x_0,y_0)\) such that \(\frac{\partial^2 L_p^{(\alpha)}}{\partial y^2}\le -c\) on \(U\). Due to \eqref{eq:J22-neg-for-IFT}, we can apply the implicit function theorem to see that there are an interval \(I\) containing \(x_0\) and a unique continuously differentiable function \(\eta:I\to(0,1)\) such that
\begin{equation}\label{eq:eta-IFT}
 \eta(x_0)=y_0\,
 \text{ and }\,
 F_{p,2}^{(\alpha)}(x,\eta(x))=0\,,
 \quad \text{for }x\in I.
\end{equation}
Shrinking \(I\) if necessary, we can choose \(\delta>0\) small enough so that \(\{(x,y):x\in I,\ |y-\eta(x)|<\delta\}\subset U\). Fix $x\in I$, then using Taylor's formula and the property in \eqref{eq:eta-IFT}, one has for \(0<|y-\eta(x)|<\delta\),
\begin{align}
 L_p^{(\alpha)}(x,y)-L_p^{(\alpha)}(x,\eta(x))
 &=(y-\eta(x))^2\int_0^1(1-s)
 \frac{\partial^2L_p^{(\alpha)}}{\partial y^2}
 \bigl(x,\eta(x)+s(y-\eta(x))\bigr)\dd s\nonumber\\
 &\le -\frac c2 (y-\eta(x))^2<0.\label{eq:Taylor-Lp}
\end{align}
Thus, for each fixed \(x\in I\), the function \(y\mapsto L_p^{(\alpha)}(x,y)\) has a strict local maximum at \(y=\eta(x)\).

We define the reduced potential by
\begin{equation*}
 \ell(x):=L_p^{(\alpha)}(x,\eta(x)).
\end{equation*}
Since \(F_{p,2}^{(\alpha)}(x,\eta(x))=0\), we deduce that
\begin{equation}\label{eq:ell-prime}
 \ell'(x)=F_{p,1}^{(\alpha)}(x,\eta(x)).
\end{equation}
The curve \(z\mapsto (x_p(z),q_p(z))\) also solves \(F_{p,2}^{(\alpha)}=0\). By the local uniqueness in \eqref{eq:eta-IFT}, for \(z\) sufficiently close to \(z_0\), we have
\begin{equation*}
 q_p(z)=\eta(x_p(z)).
\end{equation*}
Hence, plugging this into \eqref{eq:ell-prime} while taking $x=x_p(z)$, we obtain
\begin{equation}\label{eq:ell-prime-Phi}
 \ell'(x_p(z))=F_{p,1}^{(\alpha)}(x_p(z),q_p(z))=\Phi_p(z).
\end{equation}
Due to \eqref{eq:xp-prime-positive}, the sign of \(x_p(z)-x_0\) agrees with the sign of \(z-z_0\) near \(z_0\). Hence, the downward crossing assumption \eqref{eq:downward-crossing-assumption} implies that
\begin{equation*}
 \ell'(x)\begin{cases}>0\,,\quad\text{for }x<x_0\text{ close to }x_0,\\
 <0\,,\quad\text{for }x>x_0\text{ close to }x_0.
 \end{cases}
\end{equation*}
Thus, \(\ell\) has a strict local maximum at \(x_0\). Combining this with \eqref{eq:Taylor-Lp}, we get
\begin{equation*}
 L_p^{(\alpha)}(x,y)<L_p^{(\alpha)}(x_0,y_0)
\end{equation*}
for every \((x,y)\ne(x_0,y_0)\) sufficiently close to \((x_0,y_0)\).

\smallskip
\noindent\textbf{Case 2: \(z_0=1/2\).}
The one-sided sign condition gives \(\Phi_p'(1/2)\le0\). By Lemma~\ref{lem:Phi-half}, we have
\begin{equation}\label{eq:ad-lambda-nonpositive}
 \lambda_+(x_0,y_0)=\frac12\Phi_p'\!\left(\frac12\right)\le0.
\end{equation}
At the anti-diagonal equilibrium $z_0=1/2$, \eqref{e.ad.lambda_pm} and \eqref{eq:ABD} give us
\begin{equation}\label{eq:ad-lambda-D-B}
 \lambda_+(x_0,y_0)=D(y_0)+2B\!\left(\frac12\right).
\end{equation}
Applying the above display to \eqref{eq:J22-negative} shows that
\begin{equation}\label{eq:D(y_0)+B(1/2)}
 \frac{\partial F_{p,2}^{(\alpha)}}{\partial y}(x_0,y_0)
 =D(y_0)+B\!\left(\frac12\right)
 =\lambda_+(x_0,y_0)-B\!\left(\frac12\right)<0.
\end{equation}
Therefore, the implicit function theorem and the vertical maximum arguments from Case~$1$ apply and give a unique continuously differentiable function \(\eta\), a reduced potential \(\ell\), and the identity \eqref{eq:ell-prime-Phi} for \(z\) close to \(1/2\) from the right.

 Next, we check the slope of \(z\mapsto x_p(z)\) at \(1/2\). Taking $z=1/2$ and $y=y_0$ in \eqref{eq:q-prime} gives
\begin{equation*}
 q_p'\!\left(\frac12\right)=-\frac{2B(1/2)}{D(y_0)}>0.
\end{equation*}
Equations \eqref{eq:ad-lambda-nonpositive} and \eqref{eq:ad-lambda-D-B} imply \(D(y_0)\le -2B(1/2)\). Due to \eqref{def:Dp}, we have $D(y_0)<0$, and hence \(0<q_p'(1/2)\le1\). Thus, we get
\begin{equation}\label{eq:xp-prime-half-positive}
 x_p'\!\left(\frac12\right)=2-q_p'\!\left(\frac12\right)>0.
\end{equation}
This shows that for every \(x>x_0=x_p\left(\frac12\right)\) sufficiently close to \(x_0\), there is a unique \(z>z_0=1/2\) close to \(1/2\) with \(x=x_p(z)\). Applying the one-sided sign assumption to \eqref{eq:ell-prime-Phi} shows that
\begin{equation}\label{eq:right-local-max}
 \ell'(x)=\ell'(x_p(z))=\Phi_p(z)<0,
 \qquad\text{for }
 x>x_0\text{ close to }x_0.
\end{equation}
Thus
\begin{equation}\label{eq:downward-right-max}
 \ell(x)<\ell(x_0),
 \qquad x>x_0\text{ close to }x_0.
\end{equation}

It remains to approach \(x_0\) from the left. Recall the definition of $G(t)$ in \eqref{def:G}. Since \(h(1-t)=1-h(t)\), we have
\begin{equation*}
 G(1-t)=G(t)-t+\frac12.
\end{equation*}
Substitution into the expression of $L_p^{(\alpha)}(x,y)$ in \eqref{def:G} gives the symmetry
\begin{equation}\label{eq:L-symmetry}
 L_p^{(\alpha)}(x,y)=L_p^{(\alpha)}(1-y,1-x).
\end{equation}
Differentiating \(F_{p,2}^{(\alpha)}(x,\eta(x))=0\) in $x$ and evaluating at \(x=x_0\) show that
\begin{equation*}
    \frac{\partial F_{p,2}^{(\alpha)}}{\partial x}(x_0,y_0)+\frac{\partial F_{p,2}^{(\alpha)}}{\partial y}(x_0,y_0)\eta'(x_0)=0.
\end{equation*}
Plugging \eqref{eq:D(y_0)+B(1/2)} and \eqref{eq:ABD} into the display above, one has
\begin{equation}\label{eq:eta-prime-ad}
 \eta'(x_0)=-\frac{B(1/2)}{\lambda_+(x_0,y_0)-B(1/2)}>0.
\end{equation}
Since \(x_0+y_0=1\), for \(x<x_0\) sufficiently close to \(x_0\), the point \(1-\eta(x)\) is larger than \(x_0\) and close to \(x_0\). Using the symmetry \eqref{eq:L-symmetry}, the vertical maximality \eqref{eq:Taylor-Lp}, and \eqref{eq:downward-right-max}, we obtain
\begin{align*}
 \ell(x)
 &=L_p^{(\alpha)}(x,\eta(x))
 =L_p^{(\alpha)}(1-\eta(x),1-x)\\
 &\le L_p^{(\alpha)}(1-\eta(x),\eta(1-\eta(x)))
 =\ell(1-\eta(x))
 <\ell(x_0).
\end{align*}
Together with \eqref{eq:downward-right-max}, this proves that \(\ell\) has a strict local maximum at \(x_0\). Combining this with \eqref{eq:Taylor-Lp} then gives a strict local maximum of \(L_p^{(\alpha)}\) at \((x_0,y_0)\).
\end{proof}

\section{The deterministic critical point}

Define
\begin{equation}\label{eq:N-alpha}
 \calN_\alpha
 :=
 \left\{p\in\left(0,\frac{\alpha-1}{\alpha}\right):
 \exists z\in\left(\frac12,\frac34\right)\text{ such that }\Phi_p(z)<0\right\}.
\end{equation}
Let \(p_{\mathrm{ad},*}\) be the threshold from Lemma~\ref{lem:ad-stability}, and set
\begin{equation}\label{eq:A-alpha}
 \calA_\alpha:=(0,p_{\mathrm{ad},*})\cup\calN_\alpha.
\end{equation}
Define the deterministic critical point by
\begin{equation}\label{eq:pdet-alpha}
 p_{\mathrm{det}}(\alpha):=\sup\calA_\alpha.
\end{equation}

\begin{lemma}\label{lem:A-interval}
The set \(\calA_\alpha\) is downward closed. Hence, it is an interval of the form \((0,p_{\mathrm{det}}(\alpha))\), possibly including its right endpoint.
\end{lemma}

\begin{proof}
It is enough to show that \(\calN_\alpha\) is downward closed. Let \(p_2\in\calN_\alpha\), and choose \(z\in(1/2,3/4)\) with \(\Phi_{p_2}(z)<0\). We first prove that \(x_{p_2}(z)<1/2\) in a way similar to the proof of Lemma~\ref{lem:encoding}. If \(x_{p_2}(z)\ge1/2\), then by the definition of $h$ in \eqref{def:h}, we have \(h(x_{p_2}(z))\ge x_{p_2}(z)\), \(h(z)>z\), and \(z>x_{p_2}(z)\). Therefore, by an argument similar to the one used to derive \eqref{ineq.bd_Phi_p(z)>0}, we obtain
\begin{equation*}
 \Phi_{p_2}(z)=-x_{p_2}(z)+(1-p_2)h(x_{p_2}(z))+p_2h(z)>-x_{p_2}(z)+(1-p_2)x_{p_2}(z)+p_2z
 =p_2(z-x_{p_2}(z))>0,
\end{equation*}
which contradicts the assumption that $\Phi_{p_2}(z)<0$.

Now fix \(0<p_1<p_2\). By the monotonicity shown in \eqref{eq:pq-negative}, for every \(r\in[p_1,p_2]\), we have
\begin{equation*}
 q_{p_2}(z)\le q_r(z)\le q_{p_1}(z),
\end{equation*}
and hence we get
\begin{equation*}
 0<x_{p_1}(z)\le x_r(z)\le x_{p_2}(z)<\frac12.
\end{equation*}
So throughout \([p_1,p_2]\), we can apply Lemma~\ref{lem:Phi-monotone}, by which one has \(\partial_r\Phi_r(z)>0\). This gives us
\begin{equation*}
 \Phi_{p_1}(z)<\Phi_{p_2}(z)<0,
\end{equation*}
which shows that \(p_1\in\calN_\alpha\). Therefore, we can conclude that \(\calN_\alpha\) and \(\calA_\alpha\) are downward closed.
\end{proof}

\section{Proof of the phase-transition theorem}

We now prove Theorem~\ref{thm:main}. Recall that we call an equilibrium \emph{nondominating} if it is not \((0,0)\) or \((1,1)\).

\begin{proof}[Proof of Theorem~\ref{thm:main}]
We split the proof into two steps.

\smallskip
\noindent\textbf{Step 1: if \(0\le p<p_{\mathrm{det}}(\alpha)\), then \(\PP_p^{(\alpha)}(\calD)<1\).}
For \(p=0\), the two urns evolve as independent single strongly reinforced
urns. We first show that there is positive probability that urn 1 draws only
red from time \(1\) onward. Indeed, the probability that urn 1 draws red at every future step is
\[
    \prod_{n=0}^{\infty}
    \frac{W(R_0(1)+n)}
    {W(R_0(1)+n)+W(B_0(1))}= \prod_{n=0}^{\infty}\frac{1}{1+\frac{W(B_0(1))}{W(R_0(1)+n)}}.
\]
Since we have
\[
    \sum_{n=0}^{\infty}
    \frac{W(B_0(1))}{W(R_0(1)+n)}
    =
    W(B_0(1))
    \sum_{n=0}^{\infty}
    \frac{1}{(R_0(1)+n)^\alpha}
    <\infty,
\]
we can deduce that the above infinite product is positive.

Similarly, since \(B_0(2)\ge 1\), the event that urn 2 draws black at every
future step has probability
\[
    \prod_{n=0}^{\infty}
    \frac{W(B_0(2)+n)}
    {W(B_0(2)+n)+W(R_0(2))}
    >0.
\]
At \(p=0\), the two urns are independent, so the probability of the event $(x_n,y_n)\longrightarrow (0,1)$ is the product of their probabilities and is
therefore positive. Hence, we have $\mathbb P^{(\alpha)}_0(\mathcal D)<1$ .

Now assume \(0<p<p_{\mathrm{det}}(\alpha)\). By Lemma~\ref{lem:A-interval}, we have \(p\in\calA_\alpha\). If \(p<p_{\mathrm{ad},*}\in\left(0,(\alpha-1)/\alpha\right)\), then by Lemma~\ref{lem:ad-stability}, the anti-diagonal equation has a unique solution
\(u_{\alpha,p}\in(0,1/2)\), and the two off-center anti-diagonal equilibria
\[
 (u_{\alpha,p},1-u_{\alpha,p})
 \quad\text{ and }\quad
 (1-u_{\alpha,p},u_{\alpha,p})
\]
are asymptotically stable. Proposition~\ref{prop:Qin-input} shows that the process converges to either of these equilibria with positive probability. This implies that $\PP_p^{(\alpha)}(\calD)<1$.

It remains to consider \(p\in\calN_\alpha\) with \(p\ge p_{\mathrm{ad},*}\). Choose \(z_1\in(1/2,3/4)\) such that \(\Phi_p(z_1)<0\). Due to \(\Phi_p(1/2)=0\) and \(p\ge p_{\mathrm{ad},*}\), we can apply Lemmas~\ref{lem:Phi-half} and \ref{lem:ad-stability} to see that
\begin{equation*}
 \Phi_p'\!\left(\frac12\right)\ge0.
\end{equation*}

If \(\Phi_p<0\) immediately to the right of \(1/2\), then \(z_0=1/2\) is a
one-sided downward crossing. Otherwise, set
\[
    A:=\{z\in(1/2,z_1):\Phi_p(z)<0\}\text{ and } z_0:=\inf A .
\]
Then, \(z_0\in(1/2,z_1)\). Since \(\Phi_p\) is analytic and
\(\Phi_p(3/4)>0\), it is not identically zero and its zeros are isolated.
By the definition of \(z_0\), \(\Phi_p\) is nonnegative on a left neighborhood
of \(z_0\), while it is negative arbitrarily close to \(z_0\) from the
right. Therefore, we have \(\Phi_p(z_0)=0\) and for $z$ sufficiently close to $z_0$,
\[
    \Phi_p(z)\begin{cases}
        >0 \quad\text{if } z<z_0,\\
        <0 \quad\text{if } z>z_0 .
    \end{cases}
\]
Thus, in either case, there exists $z_0\in[1/2,3/4)$ satisfying
\begin{equation}\label{pf.sign_condition}
    (z-z_0)\Phi_p(z)<0
\end{equation}
for all \(z\in[1/2,3/4)\) with \(0<|z-z_0|\) sufficiently small. Thus, \(z_0\) is a downward crossing of \(\Phi_p\).

We now check that the corresponding coordinates satisfy the hypotheses of Lemma~\ref{lem:downward-crossing-stable}. This is clear for \(z_0=1/2\), where the zero is the anti-diagonal equilibrium \((u_{\alpha,p},1-u_{\alpha,p})\). If \(z_0>1/2\), then by \eqref{pf.sign_condition}, we have \(\Phi_p(z)<0\) for \(z>z_0\) close to $z_0$. The argument used in the proof of Lemma~\ref{lem:A-interval} gives \(x_p(z)<1/2\), for \(z>z_0\) close to $z_0$. Passing to the limit yields \(x_p(z_0)\le1/2\). However, the equality is impossible because otherwise,
\begin{equation*}
 \Phi_p(z_0)=-\frac12+(1-p)\frac12+ph(z_0)=p\left(h(z_0)-\frac12\right)>0,
\end{equation*}
where the last inequality is due to $z_0>1/2$ and contradicts \(\Phi_p(z_0)=0\). Hence we have \(0<x_p(z_0)<1/2<q_p(z_0)<1\), which allows us to apply Lemma~\ref{lem:downward-crossing-stable} to obtain a strict local maximum of \(L_p^{(\alpha)}\), and hence an asymptotically stable nondominating equilibrium. By Proposition~\ref{prop:Qin-input}, the urn process converges to it with positive probability, giving \(\PP_p^{(\alpha)}(\calD)<1\).

\smallskip
\noindent\textbf{Step 2: if \(p>p_{\mathrm{det}}(\alpha)\), then \(\PP_p^{(\alpha)}(\calD)=1\).}
Suppose \(p\ge(\alpha-1)/\alpha\). Let \((x,y)\in[0,1]^2\) be an equilibrium with \(x\ne y\). Assume without loss of generality that \(x<y\). Subtracting the two equilibrium equations \eqref{e.equi_eqns} gives
\[
    y-x=(1-p)\bigl(h(y)-h(x)\bigr).
\]
By Lemma~\ref{lem:hprime-monotone}, we have \(h(y)-h(x)=\int_x^y h'(t)\,dt<\alpha(y-x)\). Plugging this into the right hand side of the above display gives $y-x<(1-p)\alpha (y-x)\le y-x$, a contradiction. Therefore, every equilibrium is diagonal and the only equilibria are $(0,0),\,(1/2,1/2),\,(1,1)$.
By Proposition~\ref{prop:Qin-input}, the process almost surely converges to either \((0,0)\) or \((1,1)\), and therefore \(\mathbb P_p^{(\alpha)}(D)=1\) in this case.

It remains to consider the case $p_{\mathrm{det}}(\alpha)<p<\frac{\alpha-1}{\alpha}$.
Then \(p\notin\calA_\alpha\), and in particular \(p>p_{\mathrm{ad},*}\) by the definition of $\mathcal{A}_\alpha$ in \eqref{eq:A-alpha}. Hence, the anti-diagonal equilibrium, if exists, is unstable by Lemma~\ref{lem:ad-stability}.

We claim that there is no equilibrium in \(\Omega_+\). Suppose otherwise, then by Lemma~\ref{lem:encoding}, for some \(z_0\in(1/2,3/4)\), we have
\begin{equation*}
 \Phi_p(z_0)=0,
 \qquad
 x_p(z_0)\in\left(0,\frac12\right).
\end{equation*}
Lemma~\ref{lem:Phi-monotone} gives \(\frac{\partial}{\partial p}\Phi_p(z_0)>0\). Hence, for every \(p'<p\) sufficiently close to \(p\), one has
\begin{equation*}
 \Phi_{p'}(z_0)<0.
\end{equation*}
This implies that \(p'\in\calN_\alpha\subseteq\calA_\alpha\). Letting \(p'\uparrow p\), we get \(p\le p_{\mathrm{det}}(\alpha)\), a contradiction. Therefore, \(\Omega_+\) contains no equilibrium.

The vector field is invariant under the two maps
$(x,y)\mapsto (y,x)$ and $(x,y)\mapsto (1-y,1-x)$.
The sector \(\Omega_+\) and its images under the group generated by these maps are precisely the four open sectors obtained from the two off-diagonal rectangles after removing the anti-diagonal \(x+y=1\). Hence, the absence of equilibria in \(\Omega_+\) implies that there are no equilibria in any of these open sectors.

By Proposition~\ref{prop:Qin-input}, every equilibrium other than
$(0,0)$, $(1/2,1/2)$, and $(1,1)$ lies in one of the two off-diagonal rectangles. So the only nondominating equilibria are
\((1/2,1/2)\) and possibly, off-center anti-diagonal equilibria. The equilibrium \((1/2,1/2)\) is unstable, and the off-center anti-diagonal equilibria, if exist, are unstable in the range \(p>p_{\mathrm{ad},*}\) by Lemma~\ref{lem:ad-stability}. Thus, every nondominating equilibrium is unstable.

The process converges almost surely to an equilibrium by Proposition~\ref{prop:Qin-input}. The same proposition gives zero probability of convergence to any unstable equilibrium. Since \(\Lambda_p^{(\alpha)}\) is finite, the probability of converging to a nondominating equilibrium is zero. Hence, the limit is almost surely either \((0,0)\) or \((1,1)\), and therefore $\PP_p^{(\alpha)}(\calD)=1$.

Combining the two steps, we obtain
\begin{equation*}
 \PP_p^{(\alpha)}(\calD)
 \begin{cases}
 <1, & p<p_{\mathrm{det}}(\alpha),\\
 =1, & p>p_{\mathrm{det}}(\alpha).
 \end{cases}
\end{equation*}
The endpoint value at \(p=p_{\mathrm{det}}(\alpha)\) does not affect the infimum and supremum in \eqref{eq:critical-parameters} and \eqref{eq:tilde-critical-parameter}. Therefore, we conclude that
\begin{equation*}
 \widetilde p_\alpha=p_\alpha=p_{\mathrm{det}}(\alpha).
\end{equation*}
\end{proof}

\section{Dependence on the reinforcement exponent}

We conclude by answering two questions posed in \cite[Section~7]{Qin2024} concerning the dependence of the critical interaction parameter on the reinforcement exponent. We first determine its first-order behavior as $\alpha\downarrow1$ and then show that it is not nondecreasing in $\alpha$.

\begin{proposition}\label{prop:p-alpha-weak-limit}
Let $s_0\in(0,1)$ be the unique solution of
\begin{equation}\label{eq:s-zero-definition}
 2s_0\operatorname{arctanh}(s_0)=1.
\end{equation}
Then
\begin{equation}\label{eq:p-alpha-weak-limit}
 \lim_{\alpha\downarrow1}\frac{p_\alpha}{\alpha-1}
 =\frac{1-s_0^2}{2s_0^2}.
\end{equation}
\end{proposition}

\begin{proof}
Set $\eps=\alpha-1$ and define
\begin{equation*}
 d(t):=t(1-t)\log\frac{t}{1-t},
 \qquad t\in(0,1),
\end{equation*}
with $d(0)=d(1)=0$. Taylor expansion in the exponent gives
\begin{equation}\label{eq:h-weak-expansion}
 \sup_{t\in[0,1]}
 \left|h_{1+\eps}(t)-t-\eps d(t)\right|
 =O(\eps^2).
\end{equation}
Indeed, writing $\ell(t)=\log(t/(1-t))$, we have
\begin{equation*}
 \frac{\partial}{\partial\beta}h_\beta(t)
 =h_\beta(t)(1-h_\beta(t))\ell(t)
\end{equation*}
and
\begin{equation*}
 \frac{\partial^2}{\partial\beta^2}h_\beta(t)
 =h_\beta(t)(1-h_\beta(t))(1-2h_\beta(t))\ell(t)^2.
\end{equation*}
The second expression is uniformly bounded for $\beta\in[1,2]$ and $t\in(0,1)$. For every compact subinterval $K\subset(0,1)$, it can be verified using Taylor's expansion of $h'_\beta(t)$ in $\beta$ that
\begin{equation}\label{eq:h-prime-weak-expansion}
 \sup_{t\in K}\left|h_{1+\eps}'(t)-1-\eps d'(t)\right|
 =O(\eps^2),\qquad
 d'(t)=1+(1-2t)\log\frac{t}{1-t}.
\end{equation}

We first determine the asymptotic behavior of the anti-diagonal threshold. For $x\in(0,1/2)$, let $p_{\mathrm{ad}}^{(1+\eps)}(x)$ denote the interaction parameter for which $(x,1-x)$ is an equilibrium. By \eqref{eq:p-ad-x},
\begin{equation}\label{eq:pad-x-weak-expansion}
 p_{\mathrm{ad}}^{(1+\eps)}(x)
 =\frac{x-h_{1+\eps}(x)}{\frac12-h_{1+\eps}(x)},
\end{equation}
By \eqref{eq:h-weak-expansion}, the numerator and denominator in
\eqref{eq:pad-x-weak-expansion} satisfy
\begin{align*}
 x-h_{1+\eps}(x)&=-\eps d(x)+O(\eps^2)\\
\frac12-h_{1+\eps}(x)&=\frac12-x-\eps d(x)+O(\eps^2).
\end{align*}
If $x$ ranges over a compact subinterval of $(0,1/2)$, then
$\frac12-x$ is bounded away from zero. Dividing the two expansions above, we get
\begin{equation}\label{eq:pad-x-weak-limit}
 \frac{p_{\mathrm{ad}}^{(1+\eps)}(x)}{\eps}
 =-\frac{d(x)}{\frac{1}{2}-x}+O(\eps),
\end{equation}
uniformly on every such compact subinterval.
Let $\lambda_{+,\eps}(x)$ be the first eigenvalue in \eqref{e.ad.lambda_pm}, evaluated at $p=p_{\mathrm{ad}}^{(1+\eps)}(x)$. Combining \eqref{eq:pad-x-weak-limit}, \eqref{e.ad.lambda_pm} and \eqref{eq:h-prime-weak-expansion}, we obtain
\begin{equation}\label{eq:lambda-ad-weak-limit}
\frac{\lambda_{+,\eps}(x)}{\eps}
 =d'(x)+O(\eps),
\end{equation}
uniformly on compact subintervals of $(0,1/2)$.
The function $d'$ is strictly increasing on $(0,1/2)$. Indeed, we can compute that
\begin{equation*}
 d''(x)=\frac{1-2x}{x(1-x)}-2\log\frac{x}{1-x} >0,
 \qquad \forall x\in\left(0,\frac12\right).
\end{equation*}
Moreover, $d'(x)\to-\infty$ as $x\downarrow0$ and
$d'(x)\to1$ as $x\uparrow1/2$. Hence, $d'$ has a unique zero
$x_0\in(0,1/2)$. By \eqref{ineq.ad_lambda'>0}, the function $\lambda_{+,\eps}$ is strictly increasing,
and its unique zero is $x_*(1+\eps)$. By definition, we have
$p_{\mathrm{ad},*}(1+\eps)=p_{\mathrm{ad}}^{(1+\eps)}\bigl(x_*(1+\eps)\bigr)$. 

Since \eqref{eq:lambda-ad-weak-limit} gives
$\frac{\lambda_{+,\eps}}{\eps}\stackrel{\eps\downarrow 0}{\longrightarrow} d'$ locally uniformly on $(0,1/2)$, the corresponding zeros satisfy
\begin{equation*}
 \lim_{\eps\downarrow0}x_*(1+\eps)=x_0.
\end{equation*}

Writing $s=1-2x$, we can rewrite the expression of $d'(x)$ in \eqref{eq:h-prime-weak-expansion} as
\begin{equation}\label{eq.d'(x)_in_s}
 d'(x)=1-2s\operatorname{arctanh}(s).
\end{equation}
By the uniqueness of $s_0$ in \eqref{eq:s-zero-definition} and the uniqueness of $x_0$ as a zero of $d'$ on $(0,1/2)$, we obtain $s_0=1-2x_0$.

It now follows from \eqref{eq:pad-x-weak-limit} that
\begin{align}
 \lim_{\eps\downarrow0}
 \frac{p_{\mathrm{ad},*}(1+\eps)}{\eps}
 &=-\frac{d(x_0)}{\frac12-x_0}\nonumber\\
 &=\frac{(1-s_0^2)\operatorname{arctanh}(s_0)}{s_0}
 \stackrel{\eqref{eq:s-zero-definition}}{=}\frac{1-s_0^2}{2s_0^2}.
 \label{eq:pad-weak-limit}
\end{align}
Since $(0,p_{\mathrm{ad},*}(\alpha))\subseteq\calA_\alpha$ and $p_{\det} (\alpha)=\sup \calA_\alpha $, we can apply Theorem~\ref{thm:main} to get the lower bound in \eqref{eq:p-alpha-weak-limit}.

For the upper bound, we  suppose that
\begin{equation*}
 \limsup_{\alpha\downarrow1}\frac{p_\alpha}{\alpha-1}>\frac{1-s_0^2}{2s_0^2}.
\end{equation*}
Since $p_\alpha=p_{\mathrm{det}}(\alpha)\le(\alpha-1)/\alpha$, we may choose $c\in\left(\frac{1-s_0^2}{2s_0^2},1\right)$ below this upper limit and a sequence $\eps_n\downarrow0$ such that
\begin{equation*}
 p_n:=c\eps_n<p_{1+\eps_n}.
\end{equation*}
By Theorem~\ref{thm:main}, the above display implies that
\begin{equation*}
 \PP_{p_n}^{(1+\eps_n)}(\calD)<1.
\end{equation*}

By Proposition~\ref{prop:Qin-input}, the process converges almost surely
to one of finitely many equilibria. Hence, for every $n$, there exists a
nondominating equilibrium $(x^{(n)},y^{(n)})$ such that
\begin{equation*}
 \PP_{p_n}^{(1+\eps_n)}\left(\lim_{k\to\infty}(x_k,y_k)=\left(x^{(n)},y^{(n)}\right)\right)>0.
\end{equation*}
Proposition~\ref{prop:Qin-input}(v) therefore gives
\begin{equation}\label{ineq.lambda_+(xn,yn)<=0}
 \lambda_+\bigl(x^{(n)},y^{(n)}\bigr)\le0.
\end{equation}

By \eqref{eq:pad-weak-limit} and the definition of $p_n$, one has $p_n>p_{\mathrm{ad},*}(1+\eps_n)$ for all sufficiently large $n$. Lemma~\ref{lem:ad-stability} and the instability of $(1/2,1/2)$ show that $(x^{(n)},y^{(n)})$ is neither anti-diagonal nor central. Since the vector field is invariant under $(x,y)\mapsto(y,x)$ and $(x,y)\mapsto(1-y,1-x)$,
we may assume that
\begin{equation*}
\left(x^{(n)},y^{(n)}\right)\in\Omega_+.
\end{equation*}

Passing to a subsequence, we can suppose that $x^{(n)}\to x$ and $y^{(n)}\to y$, and set $z=(x+y)/2$. Then
\begin{equation*}
 0\le x\le\frac12\le y\le1,
 \qquad
 x+y\ge1.
\end{equation*}
Dividing the equilibrium equations by $\eps_n$ and using \eqref{eq:h-weak-expansion}, we obtain
\begin{equation}\label{eq:limiting-weak-equilibrium}
 d(x)+c(z-x)=0,
 \qquad
 d(y)+c(z-y)=0.
\end{equation}
Recall that $d(0)=d(1)=d(1/2)=0$. The above equations imply that $x\neq0$ and $y\neq1$. Moreover, if $x=1/2$ or $y=1/2$, they give $x=y=1/2$. Thus, either the limit is $(1/2,1/2)$ or $0<x<1/2<y<1$.

In both cases, there exist
$\delta\in(0,1/2)$ and $N\in\mathbb N$ such that, for every $n\ge N$,
\begin{equation*}
 x^{(n)},\quad y^{(n)},\quad
 \frac{x^{(n)}+y^{(n)}}2
 \in[\delta,1-\delta].
\end{equation*}
This implies that the uniform expansion
\eqref{eq:h-prime-weak-expansion} applies at $x^{(n)}, y^{(n)}, \frac{x^{(n)}+y^{(n)}}2$. Therefore, we can use the expression for the Jacobian matrix in \eqref{eq:J-ABD} and \eqref{eq:h-prime-weak-expansion} to obtain
\begin{equation}\label{eq:limiting-weak-jacobian}
 \frac1{\eps_n}DF_{p_n}^{(1+\eps_n)}\bigl(x^{(n)},y^{(n)}\bigr)\longrightarrow
 \begin{pmatrix}
 d'(x)-\dfrac c2&\dfrac c2\\[2mm]
 \dfrac c2&d'(y)-\dfrac c2
 \end{pmatrix}.
\end{equation}
Due to \eqref{ineq.lambda_+(xn,yn)<=0}, we see that for each $n$, the Jacobian matrix $DF_{p_n}^{(1+\eps_n)}\bigl(x^{(n)},y^{(n)}\bigr)$ is negative semidefinite. Hence, the limiting matrix in \eqref{eq:limiting-weak-jacobian} is also negative semidefinite. At $(x,y)=(1/2,1/2)$, the limiting matrix has eigenvalue $1$ in the direction $(1,1)$, which contradicts its negative semidefiniteness.

Utilizing the two equations in \eqref{eq:limiting-weak-equilibrium} and the fact that $d(t)=d(1-t)$ for all $t\in(0,1)$ gives
\begin{equation}\label{eq:weak-equal-level}
 d(1-x)=d(y),
 \qquad
 c=\frac{2d(y)}{y-x}.
\end{equation}

Suppose first that $y>1-x$. Set
\begin{equation*}
 a:=\frac12\log\frac{1-x}{x},
 \qquad
 b:=\frac12\log\frac{y}{1-y}.
\end{equation*}
Then we have $\tanh(a)=1-2x$, $\tanh(b)=2y-1$, so $0<a<b$. Moreover, we 
\begin{equation*}
 d(1-x)=\frac{a}{2}\operatorname{sech}^2(a),
 \qquad
 d(y)=\frac{b}{2}\operatorname{sech}^2(b),
\end{equation*}
so we can rewrite \eqref{eq:weak-equal-level} as
\begin{equation}\label{eq:weak-logit-equal-level}
 a\operatorname{sech}^2(a)
 =b\operatorname{sech}^2(b),
 \qquad
 c=
 \frac{2a\operatorname{sech}^2(a)}
 {\tanh(a)+\tanh(b)}.
\end{equation}

Let $a_0:=\operatorname{arctanh}(s_0)$. Due to \eqref{eq:s-zero-definition} that $2a_0\tanh(a_0)=2s_0\operatorname{arctanh}(s_0)=1$
and
\begin{equation*}
 \frac{\dd}{\dd r}
 \left(r\operatorname{sech}^2(r)\right)
 =
 \operatorname{sech}^2(r)
 \left(1-2r\tanh(r)\right),
\end{equation*}
we see that $r\mapsto r\operatorname{sech}^2(r)$ strictly increases on $(0,a_0]$ and strictly decreases on $(a_0,\infty)$.
The first identity in \eqref{eq:weak-logit-equal-level} and $a<b$ imply that $a<a_0<b$.
Taking logarithms in the first identity in
\eqref{eq:weak-logit-equal-level} gives
\begin{equation*}
 0
 =
 \int_a^b
 \left(\frac1r-2\tanh(r)\right)\dd r.
\end{equation*}
For $r>0$, taking derivatives of the above integrand, we obtain
\begin{align*}
 \frac{\dd}{\dd r}
 \left(\frac1r-2\tanh(r)\right)
 &=-\frac1{r^2}-2\operatorname{sech}^2(r)<0,\\
 \frac{\dd^2}{\dd r^2}
 \left(\frac1r-2\tanh(r)\right)
 &=\frac2{r^3}
 +4\operatorname{sech}^2(r)\tanh(r)>0.
\end{align*}
Thus, the integrand is strictly decreasing and strictly convex.
We can apply the Hermite--Hadamard inequality to get
\begin{equation*}
 \frac2{a+b}-2\tanh\left(\frac{a+b}{2}\right)<\frac1{b-a}\int_a^b\left(\frac1r-2\tanh(r)\right)\dd r
 =0.
\end{equation*}
By the facts that 
$\frac1{a_0}-2\tanh(a_0)=0$
and $r\mapsto 1/r-2\tanh(r)$ is strictly decreasing, we obtain
\begin{equation}\label{ineq.a+b>2a_0}
 \frac{a+b}{2}>a_0.
\end{equation}

The first identity in \eqref{eq:weak-logit-equal-level} also gives
\begin{equation*}
 \frac{\cosh(b)}{\cosh(a)}=\sqrt{\frac ba},
\end{equation*}
which combined with the second equation in
\eqref{eq:weak-logit-equal-level} implies that
\begin{equation*}
 c
 =\frac{2a\cosh(b)}
 {\cosh(a)\sinh(a+b)}
 =\frac{2\sqrt{ab}}{\sinh(a+b)}<\frac{a+b}{\sinh(a+b)}
 \stackrel{\eqref{ineq.a+b>2a_0}}{<}\frac{2a_0}{\sinh(2a_0)}
 =\frac{1-s_0^2}{2s_0^2}.
\end{equation*}
This contradicts the assumption that $c>\frac{1-s_0^2}{2s_0^2}$.

It remains to consider the case $y=1-x$. Recall that we have let $s=1-2x$, so we also have $s=2y-1$. The limiting matrix in \eqref{eq:limiting-weak-jacobian} has an eigenvector $(1,1)$ corresponding with the eigenvalue
\begin{equation*}
 d'(x)=d'(y)\stackrel{\eqref{eq.d'(x)_in_s}}{=}1-2s\operatorname{arctanh}(s),
\end{equation*}
which by negative semidefiniteness gives $s\ge s_0$. For the expression of $c$ in \eqref{eq:weak-equal-level}, we can rewrite it in terms of $s$ as
\begin{equation*}
 c=\frac{(1-s^2)\operatorname{arctanh}(s)}{s}.
\end{equation*}
Taking derivative of the above right hand side gives
\begin{equation*}
 \frac{s-(1+s^2)\operatorname{arctanh}(s)}{s^2}<0,
\end{equation*}
so the right hand side function strictly decreases on $(0,1)$.
Consequently, we obtain $c\le c_0$, which contradicts our choice of $c$. This proves the upper bound and completes the proof.
\end{proof}

\begin{proposition}\label{prop:p-alpha-not-monotone}
The map $\alpha\mapsto p_\alpha$ is not nondecreasing. More precisely,
\begin{equation}\label{eq:p-five-halves-p-three}
 p_{5/2}>\frac{31}{100}>p_3.
\end{equation}
\end{proposition}

\begin{proof}
By Theorem~\ref{thm:main}, it is enough to prove
\begin{equation}\label{eq:pdet-separation}
 p_{\mathrm{det}}\left(\frac52\right)>\frac{31}{100}>
 p_{\mathrm{det}}(3).
\end{equation}

\smallskip
\noindent\textbf{The lower bound for $\alpha=\frac52$.}
In this part, $h$, $\widetilde F_p$, and $\Phi_p$ are understood with $\alpha=\frac52$. Set
\begin{equation*}
 p_0:=\frac{31}{100},
 \qquad
 z_0:=\frac{21}{40},
 \qquad
 q_-:=\frac{2141}{2500},
 \qquad
 q_+:=\frac{1713}{2000}.
\end{equation*}
For rational $t,b\in(0,1)$, squaring the positive quantities in the definition of $h=h_{5/2}$ gives
\begin{equation}\label{eq:five-halves-certificate}
\begin{aligned}
 h(t)>b
 &\quad\Longleftrightarrow\quad
 \left(\frac{1-t}{t}\right)^5
 <
 \left(\frac{1-b}{b}\right)^2,\\
 h(t)<b
 &\quad\Longleftrightarrow\quad
 \left(\frac{1-t}{t}\right)^5
 >
 \left(\frac{1-b}{b}\right)^2.
\end{aligned}
\end{equation}
This criterion, followed only by clearing positive denominators, gives
\begin{equation}\label{eq:five-halves-rational-bounds}
\begin{gathered}
 \frac{56222}{10^5}<h(z_0)<\frac{56223}{10^5},\\
 h(q_-)>\frac{98861}{10^5},
 \qquad
 h(q_+)<\frac{98865}{10^5},
 \qquad
 h\left(\frac{121}{625}\right)<\frac{2747}{10^5}.
\end{gathered}
\end{equation}
Substitution of \eqref{eq:five-halves-rational-bounds} into the definition of $\widetilde F_{p_0}$ yields
\begin{equation*}
 \widetilde F_{p_0}(q_-,z_0)>0
 \qquad\text{and}\qquad
 \widetilde F_{p_0}(q_+,z_0)<0.
\end{equation*}
Lemma~\ref{lem:upper-branch} therefore gives
\begin{equation*}
 q_-<q_{p_0}(z_0)<q_+,
 \qquad
 \frac{387}{2000}<x_{p_0}(z_0)<\frac{121}{625}.
\end{equation*}
Using once more that $h$ is increasing and applying \eqref{eq:five-halves-rational-bounds}, we obtain
\begin{equation*}
 \Phi_{p_0}(z_0)
 <-\frac{387}{2000}
 +\frac{69}{100}\frac{2747}{10^5}
 +\frac{31}{100}\frac{56223}{10^5}
 =-\frac{159}{625000}<0.
\end{equation*}
The implicit function theorem and \eqref{def:Dp} show that $p\mapsto q_p(z_0)$, and hence $p\mapsto\Phi_p(z_0)$, is continuous near $p_0$. Thus, $\Phi_p(z_0)<0$ for some $p>p_0$ sufficiently close to $p_0$. Since $p_0<3/5=(\alpha-1)/\alpha$, we may also require $p<3/5$. It follows that $p\in\calN_{5/2}$, and therefore we get
\begin{equation}\label{eq:pdet-five-halves-lower}
 p_{\mathrm{det}}\left(\frac52\right)>p_0.
\end{equation}

\smallskip
\noindent\textbf{The upper bound for $\alpha=3$.}
In this part, $h=h_3$ and $p_0=31/100$. Let $p\in(0,2/3)$ and let $(x,y)\in\Omega_+$ be an equilibrium. Set
\begin{equation*}
 v:=x+y-2xy.
\end{equation*}
Since $0<x<1/2<y<1$, we have $1/2<v<1$. Substituting
\begin{equation*}
 h_3(t)=\frac{t^3}{t^3+(1-t)^3}
\end{equation*}
into the two equations $F_p^{(3)}(x,y)=0$, a few basic operations give us
\begin{equation}\label{eq:R-three-zero}
 R(p,v)=0,
\end{equation}
where we define
\begin{equation}\label{eq:R-three}
 R(p,v)
 :=9v^3p^2+2v(27v^2-15v+4)p
 +4(2v-1)(9v^2-12v+2).
\end{equation}
At $p=p_0$, one has
\begin{equation}\label{eq:R-three-at-p0}
 10^4R(p_0,v)
 =896049v^3-1413000v^2+664800v-80000.
\end{equation}
The cubic on the right-hand side of \eqref{eq:R-three-at-p0} has negative discriminant and is positive at $v=1/2$. Since its leading coefficient is positive, its unique real zero lies below $1/2$. Consequently, we have
\begin{equation}\label{eq:R-three-positive}
 R(p_0,v)>0,
 \qquad
 \text{for every }v\in\left[\frac12,1\right].
\end{equation}
Moreover, taking partial derivative of $R$ in $p$ gives
\begin{equation*}
 \partial_pR(p,v)
 =18v^3p+2v(27v^2-15v+4)>0,
\end{equation*}
for $p\ge0$ and $v\in[1/2,1]$. By \eqref{eq:R-three-positive} and compactness, there exists some $\delta\in(0,p_0)$ such that $R(p,v)>0$ for every $p\in[p_0-\delta,p_0]$ and $v\in[1/2,1]$. The strict positivity of $\partial_pR$ then gives
\begin{equation}\label{eq:R-three-uniform-positive}
 R(p,v)>0,
 \qquad
 \text{for }p\ge p_0-\delta
 \text{ and }v\in\left[\frac12,1\right].
\end{equation}
It follows from \eqref{eq:R-three-zero} that every equilibrium in $\Omega_+$ with $p\in(0,2/3)$ satisfies $p<p_0-\delta$.

If $p\in\calN_3$, choose $z\in(1/2,3/4)$ such that $\Phi_p(z)<0$. Lemma~\ref{lem:endpoint-sign} and continuity give a zero of $\Phi_p$ in $(z,3/4)$. By Lemma~\ref{lem:encoding}, this zero gives an equilibrium in $\Omega_+$. Hence,
\begin{equation}\label{eq:N-three-strict-upper}
 \sup\bigl(\calN_3\cup\{0\}\bigr)
 \le p_0-\delta<p_0.
\end{equation}

It remains to control the anti-diagonal threshold. Using \eqref{eq:u-alpha-p} to determine $u_{3,p_0}$ and then substituting into \eqref{e.ad.lambda_pm} gives
\begin{equation*}
 \lambda_+(u_{3,p_0},1-u_{3,p_0})
 =\frac{317}{9200}>0.
\end{equation*}
Therefore, the uniqueness and sign characterization in Lemma~\ref{lem:ad-stability} imply
\begin{equation}\label{eq:pad-three-upper}
 p_{\mathrm{ad},*}(3)<p_0.
\end{equation}
By \eqref{eq:A-alpha} and \eqref{eq:pdet-alpha}, one has
\begin{equation*}
 p_{\mathrm{det}}(3)
 =
 \sup\left((0,p_{\mathrm{ad},*}(3))\cup\calN_3\right).
\end{equation*}
Hence, the strict bounds \eqref{eq:N-three-strict-upper} and \eqref{eq:pad-three-upper} give
\begin{equation}\label{eq:pdet-three-upper}
 p_{\mathrm{det}}(3)<p_0.
\end{equation}

Combining \eqref{eq:pdet-five-halves-lower} and \eqref{eq:pdet-three-upper} proves \eqref{eq:pdet-separation}, which concludes the proof.
\end{proof}

\section*{Acknowledgements}
During the preparation of this paper, the author used ChatGPT (OpenAI) as a tool for discussing proof strategies and checking intermediate arguments. All mathematical statements and proofs in the paper were independently verified by the author, who takes full responsibility for the contents of the manuscript.

\bibliographystyle{plain}
\bibliography{ref}

\end{document}